\documentclass[twoside]{article}

\usepackage{PRIMEarxiv}

\usepackage[utf8]{inputenc} 
\usepackage[T1]{fontenc}    
\usepackage{hyperref}       
\usepackage{url}            
\usepackage{booktabs}       
\usepackage{amsfonts}       
\usepackage{nicefrac}       
\usepackage{microtype}      
\usepackage{lipsum}
\usepackage{mathtools}
\usepackage{multicol}
\usepackage{fancyhdr}       
\usepackage{graphicx}       
\graphicspath{{media/}}     
\usepackage{relsize}
\usepackage{amsmath}
\usepackage{amsthm}
\usepackage{amssymb}
\usepackage{tikz}
\usepackage{tikz-cd}
\usetikzlibrary{calc}
\usepackage{faktor}
\usepackage{adjustbox}
\usepackage{yhmath}

\usetikzlibrary{arrows,calc,matrix}

  \newtheorem{theorem}{Theorem}[section]
	\newtheorem{proposition}[theorem]{Proposition}
	\newtheorem{conjecture}[theorem]{Conjecture}
	\newtheorem{problem}[theorem]{Problem}
	
	\newtheorem{lemma}[theorem]{Lemma}
	\newtheorem{corollary}[theorem]{Corollary}
	\theoremstyle{definition}
	\newtheorem*{claim*}{Claim}
	\newtheorem{remark}[theorem]{Remark}
	
	\newtheorem{definition}[theorem]{Definition}
	
	\newtheorem{example}[theorem]{Example}
	
	\theoremstyle{remark}

   \DeclareMathOperator {\lk}{lk} 
    \DeclareMathOperator {\st}{st}

  \DeclareMathOperator {\Out}{Out} 
  \DeclareMathOperator {\Aut}{Aut} 

\usepackage{flowchart}
\usetikzlibrary{arrows}

\newcommand{\abs}[1]{\left|#1\right|}
\usepackage{scalerel,stackengine}
\stackMath
\newcommand\reallywidehat[1]{%
	\savestack{\tmpbox}{\stretchto{%
			\scaleto{%
				\scalerel*[\widthof{\ensuremath{#1}}]{\kern-.6pt\bigwedge\kern-.6pt}%
				{\rule[-\textheight/2]{1ex}{\textheight}}
			}{\textheight}%
		}{0.5ex}}%
	\stackon[1pt]{#1}{\tmpbox}%
}
\newcommand\restr[2]{{
		\left.\kern-\nulldelimiterspace 
		#1 
		\right|_{#2} 
}}
  
\title{On the $\ell^2$-Betti numbers and algebraic fibring of the (outer) automorphism group of a right-angled Artin groups}
\author{Marcos Escartín-Ferrer}

\begin{document}

 \maketitle

\begin{abstract} 
We compute the first $\ell^2$-Betti number of the automorphism and outer automorphism groups of arbitrary right-angled Artin groups (RAAGs), providing a complete characterization of when it is non-zero. We also analyse the algebraic fibring of the pure symmetric automorphism groups $\mathrm{PSA}(A_\Gamma)$ and $\mathrm{PSO}(A_\Gamma)$ and the virtual algebraic fibring of $\Out(A_\Gamma)$ in the case when $A_\Gamma$ admits no non-inner partial conjugation. In the transvection-free case, we show that $\beta_1^{(2)}(\Out(A_\Gamma)) = 0$ if and only if $\Out(A_\Gamma)$ virtually fibres.
\end{abstract}
\keywords{Right-angled Artin groups \and $\ell^2$-Betti numbers \and Automorphism groups \and Algebraic fibring}
\section*{Acknowledgements}
The author is partially supported by the Departamento de Ciencia, Universidad y Sociedad del Conocimiento del Gobierno de Aragón (grant code: E22-23R: “Álgebra y Geometría”), and by the Spanish Government PID2021-126254NBI00. The author
	would like to thank  Sam Fisher, Dawid Kielak and Ric Wade for conversations during a short stay in Oxford.
\section{Introduction}\label{Intro}
The $\ell^2$-Betti numbers of a group were first introduced by Atiyah (cf. \cite{Atilla}) in the context of free, cocompact group actions on manifolds. Cheeger and Gromov later extended the notion to arbitrary countable groups (cf. \cite{Cheeger}). These invariants serve as powerful tools for capturing asymptotic properties of groups. For a countable group $G$, the $\ell^2$-Betti numbers are denoted by $\beta_k^{(2)}(G)$, where $k$ is a non-negative integer. A rigorous definition and thorough exposition can be found in \cite{Kammeyer}.

Given a simplicial graph $\Gamma$, i.e., a finite graph with no loops or multiple edges, the associated \textbf{right-angled Artin group} (RAAG) is defined as the group generated by the vertex set $V(\Gamma)$, with relations $uv = vu$ for every edge $e = \lbrace u, v\rbrace \in E(\Gamma)$. This family of groups interpolates between free groups and free abelian groups, providing a flexible framework for exploring group-theoretic phenomena.

The (outer) automorphism groups of RAAGs serve as a bridge between $\mathrm{GL}_n(\mathbb{Z})$ and the automorphism and outer automorphism groups of free groups, $\Aut(F_n)$ and $\Out(F_n)$. Laurence showed in \cite{Laurence} that $\Aut(A_\Gamma)$, and hence $\Out(A_\Gamma)$, is finitely generated. A key concept for understanding its generating set is the \textbf{domination relation}, introduced by Servatius (cf. \cite{Servatius}). This relation characterizes when the \textbf{transvection} $\tau_w^v$, defined by sending the generator $w$ to $wv$ and fixing all others, yields an automorphism of $A_\Gamma$. Specifically, we say that $w$ is dominated by $v$, written $w \leq v$, if $\lk(w) \subset \st(v)$ (every vertex adjacent to $w$ is also adjacent to or equal to $v$). It is straightforward to verify that $\tau_w^v \in \Aut(A_\Gamma)$ if and only if $w \leq v$. This relation defines a preorder on the vertex set, giving rise to equivalence classes of vertices.

Another fundamental type of generators are the \textbf{partial conjugations}. Given a vertex $v \in V(\Gamma)$ and a connected component $C$ of the subgraph $\Gamma \setminus \st(v)$, the partial conjugation $\pi_C^v$ is the automorphism that sends each generator $u \in V(C)$ to $vuv^{-1}$ and fixes all other generators.

The $\ell^2$-Betti numbers of $\mathrm{GL}_n(\mathbb{Z})$ are well understood (see Theorem~\ref{BettiGL(Z)}), whereas the $\ell^2$-Betti numbers of $\Aut(F_n)$ and $\Out(F_n)$ remain largely unknown. Nonetheless, several partial results are known:
\begin{itemize}
	\item $\beta_1^{(2)}(\Out(F_n)) = 0$ for all $n \neq 2$, and $\beta_1^{(2)}(\Aut(F_n)) = 0$ for all $n \geq 0$ (cf.  Gaboriau–Nôus \cite{Gaboriau}).
	\item $\beta_2^{(2)}(\Out(F_n)) = 0$ for all $n \geq 5$ (cf. Abért–Gaboriau \cite{Abert}).
	\item $\beta_{2n - 3}^{(2)}(\Out(F_n)) > 0$ and $\beta_{2n - 2}^{(2)}(\Aut(F_n)) > 0$ for all $n \geq 2$ (cf. Gaboriau–Nôus \cite{Gaboriau}).
\end{itemize}
Here, we broaden the scope to arbitrary RAAGs and consider the following:
\begin{problem}
	Given a RAAG $A_\Gamma$, compute the $\ell^2$-Betti numbers of $\Aut(A_\Gamma)$ and $\Out(A_\Gamma)$, or at least determine whether they vanish.
\end{problem}

A central object in our study of the $\ell^2$-Betti numbers is the pure symmetric automorphism group. For a RAAG $A_\Gamma$, the \textbf{pure symmetric automorphism group}, denoted $\mathrm{PSA}(A_\Gamma)$, is the subgroup of $\Aut(A_\Gamma)$ generated by all partial conjugations. The corresponding quotient in $\Out(A_\Gamma)$ is called the \textbf{pure symmetric outer automorphism group}, denoted $\mathrm{PSO}(A_\Gamma)$. Day and Wade proved in \cite{Day 4} that every RAAG arises as the pure symmetric outer automorphism group of another RAAG. Moreover, they showed that if $\Gamma \setminus \st(v)$ has at most two connected components for every vertex $v \in V(\Gamma)$, then there exists a graph $\Theta$ such that $\mathrm{PSO}(A_\Gamma)$ is isomorphic to $A_\Theta$, see Section \ref{SymetricAUT} for the explicit construction of the graph $\Theta$.

The main result of this paper is the computation of the first $\ell^2$-Betti number of $\Aut(A_\Gamma)$ and $\Out(A_\Gamma)$ for arbitrary RAAGs.
\begin{theorem}\label{Theorem 1.1}
	Let $A_\Gamma$ be a RAAG. Then, $\beta_1^{(2)}(\Aut(A_\Gamma)) > 0$ if and only if $A_\Gamma \cong \mathbb{Z}^2$ and $\beta_1^{(2)}(\Out(A_\Gamma)) > 0$ if and only if one of the following conditions holds:
		\begin{enumerate}
			\item There are no non-inner partial conjugations, and there are exactly two transvections lying in the same equivalence class of two elements.
			\item There are no transvections, $\Gamma \setminus \st(v)$ has at most two connected components for every $v \in V(\Gamma)$, and $\mathrm{PSO}(A_\Gamma)$ is isomorphic to a RAAG $A_\Theta$ where the graph $\Theta$ is disconnected.
		\end{enumerate}
\end{theorem}
The higher $\ell^2$-Betti numbers of the (outer) automorphism groups of RAAGs remain largely mysterious in general.

A group $G$ is said to be \textbf{algebraically fibred} if there exists a surjective homomorphism $\varphi: G \to \mathbb{Z}$ with finitely generated kernel. This notion is closely linked to the first $\ell^2$-Betti number: if $G$ is virtually algebraically fibred, then $\beta_1^{(2)}(G) = 0$ (see Proposition \ref{VanishingBetti}.3). Conversely, Kielak proved that for virtually RFRS (residually finite rationally solvable) groups, vanishing of the first $\ell^2$-Betti number implies that the group is virtually algebraically fibred (cf. \cite{Kielak}). RAAGs are known to be virtually RFRS.

We completely determine the algebraic fibring properties of the pure symmetric automorphism groups $\mathrm{PSA}(A_\Gamma)$ and $\mathrm{PSO}(A_\Gamma)$:

\begin{theorem}\label{Theorem 1.5}
	Let $A_\Gamma$ be a RAAG. Then:
	\begin{enumerate}
		\item $\mathrm{PSA}(A_\Gamma)$ fibres if and only if it virtually fibres, which occurs if and only if $A_\Gamma \not\cong \mathbb{Z}^n$ or $\mathbb{Z}^n \ast \mathbb{Z}^m$ for any $n, m \geq 1$.
		\item $\mathrm{PSO}(A_\Gamma)$ fibres if and only if it virtually fibres, which occurs if and only if either:
		\begin{itemize}
			\item There exists $v \in V(\Gamma)$ such that $\Gamma \setminus \st(v)$ has at least three connected components, or
			\item For every $v \in V(\Gamma)$, the graph $\Gamma \setminus \st(v)$ has at most two connected components, and $\mathrm{PSO}(A_\Gamma)$ is isomorphic to $A_\Theta$ with $\Theta$ connected.
		\end{itemize}
	\end{enumerate}
\end{theorem}
We conjecture that Kielak’s theorem remains valid when the assumption of being RFRS is weakened to being residually torsion-free nilpotent.

\begin{conjecture}\label{Conjecture 1.4}
	Let $G$ be a virtually residually torsion-free nilpotent group. Then, $\beta_1^{(2)}(G) = 0$ if and only if $G$ is virtually algebraically fibred.
\end{conjecture}

In the setting of automorphism groups of RAAGs, the \textbf{Torelli subgroup} $\mathrm{IA}_\Gamma$ (defined as the kernel of the canonical homomorphism $\Out(A_\Gamma) \to \mathrm{GL}_n(\mathbb{Z})$ induced by the action on the abelianization of $A_\Gamma$) is a prominent example of a residually torsion-free nilpotent group (cf. \cite{Wade}).

When $A_\Gamma$ is \textbf{transvection-free}, i.e. if $u,v\in V(\Gamma)$ are such that $\lk(u)\subset\st(v)$ then $u=v$, the Torelli subgroup is of finite index in $\Out(A_\Gamma)$. As a consequence of Theorems~\ref{Theorem 1.1} and \ref{Theorem 1.5}, we verify that Conjecture~\ref{Conjecture 1.4} holds for $\Out(A_\Gamma)$ in this case.

\begin{corollary}
	Let $A_\Gamma$ be a transvection-free RAAG. Then, $\beta_1^{(2)}(\Out(A_\Gamma)) = 0$ if and only if $\Out(A_\Gamma)$ is virtually algebraically fibred.
\end{corollary}

We also identify two conditions on the set of transvections that guarantee the virtual algebraic fibering of $\Out(A_\Gamma)$, and we provide a complete characterization of virtual algebraic fibering in the case when there are no non-inner partial conjugations.

\begin{theorem}\label{Theorem 1.7}
	Let $A_\Gamma$ be a RAAG. Then $\Out(A_\Gamma)$ is virtually algebraically fibred if at least one of the following conditions holds:
	\begin{enumerate}
		\item There exist distinct vertices $u, v \in V(\Gamma)$ such that $[u] =\lbrace u\rbrace$, $[v]=\lbrace v\rbrace$, $v \leq u$, and there is no vertex $w \neq u,v$ with $v \leq w \leq u$.
		\item There exist at least two equivalence classes of vertices containing exactly two elements.
	\end{enumerate}
Moreover, if $A_\Gamma$ admits no non-inner partial conjugations, then the converse holds.
\end{theorem}
The paper is structured as follows. In section $2$, we present the necessary definitions and preliminary results that will be used throughout the text. Section $3$ is devoted to the  study of virtual fibring, where we prove Theorems \ref{Theorem 1.5} and \ref{Theorem 1.7}. In section $4$, we study the $\ell^2$-Betti numbers of $\Aut(A_\Gamma)$ and $\Out(A_\Gamma)$, proving Theorem \ref{Theorem 1.1}. In the final section, we present several examples of RAAGs for which $\Aut(A_\Gamma)$ and $\Out(A_\Gamma)$ exhibit distinct fibring properties and $\ell^2$-Betti number behaviour.
\section{Preliminaries}
\subsection{$\ell^2$-Betti numbers}
Given a countable discrete group $G$ the $\ell^2$-Betti numbers $\beta_i^{(2)}(G)$ are powerful asymptotic invariants that capture geometric, topological, and analytic properties of the group. Since their introduction, $\ell^2$-Betti numbers have found significant applications in geometry, especially in the study of manifolds, in ergodic theory, operator algebras, and numerous areas of group theory. For a comprehensive introduction and a survey of applications, we refer the reader to the monographs by Kammeyer and Lück \cite{Kammeyer, Luck}.

In this section, we collect the main properties of $\ell^2$-Betti numbers that will be used throughout the paper.

We begin with a fundamental observation: the zeroth $\ell^2$-Betti number distinguishes between finite and infinite groups.
\begin{lemma}[\cite{Kammeyer}, Exercise 4.4.1]\label{LemmaZeroth}
	Let $G$ be a countable group. Then:
	\begin{enumerate}
		\item If $G$ is finite, then $\beta_0^{(2)}(G)=\frac{1}{\abs{G}}$ and $\beta_i^{(2)}(G)=0~\forall~i\geq1$.
		\item If $G$ is infinite, then $\beta_0^{(2)}(G)=0$.
\end{enumerate}
\end{lemma}
Another key feature is that $\ell^2$-Betti numbers behave predictably with respect to finite-index subgroups.
\begin{proposition}[\cite{Kammeyer}, Theorem 4.15 (iii)]\label{FiniteIndexBetti}
	Let $H\leq G$ be a subgroup of finite index. Then, for all $n\geq 0$: $$\beta^{(2)}_n(G)=\frac{\beta^{(2)}_n(H)}{\abs{G:H}}$$
\end{proposition}

We now record several important vanishing results.

\begin{proposition}\label{VanishingBetti}
	Let $G$ be a countable group. Then:
	\begin{enumerate}
		\item If $H\trianglelefteq G$ and $\beta_i^{(2)}(H)=0$ for all $i=0,\dots,n$ then $\beta_i^{(2)}(G)=0$ for all $i=0,\dots,n$.
		\item All \(\ell^2\)-Betti numbers of infinite amenable groups vanish.
		\item Let $1\to H\to G\to K\to 1$ be a short exact sequence of countable infinite groups such that $\beta_1^{(2)}(H)<\infty$. Then, $\beta_1^{(2)}(G)=0$.
	\end{enumerate}
\end{proposition}
\begin{proof}
	 Follows from \cite[Theorem 7.2 (1), (2), (7)]{Luck}.
\end{proof}
There is also a Künneth-type formula that describes the behavior of $\ell^2$-Betti numbers under direct products:
\begin{theorem}[\cite{Kammeyer}, Theorem 4.15 (i)] \label{KunnethBetti}
	Let $G_1$ and $G_2$ be two countable groups. Then, for all $n\geq0$:
	$$\beta^{(2)}_n(G_1\times G_2)=\sum\limits_{i=0}^n\beta_i^{(2)}(G_1)\beta_{n-i}^{(2)}(G_2)$$
\end{theorem}

An important well-known fact that we will use is that the first $\ell^2$-Betti number of a finitely generated group is always finite. Since we could not find a proof in the literature stated, we include a short argument. The proof relies only on standard properties of $\ell^2$-homology and von Neumann dimension (cf.\ \cite{Luck}).

\begin{lemma}
	Let $G$ be an infinite finitely generated group. Then $\beta_1^{(2)}(G)\le d(G)-1,$
	where $d(G)$ denotes the minimal number of generators of $G$.
\end{lemma}

\begin{proof}
	If $d=d(G)$, we may construct a $K(G,1)$-complex by taking a bouquet of $d$ circles and attaching higher-dimensional cells. Passing to the universal cover, we obtain a cellular chain complex of Hilbert $\mathcal{N}(G)$-modules of the form
	$$
	\cdots \xrightarrow{\partial_2} \ell^2(G)^d \xrightarrow{\partial_1} \ell^2(G) \to 0.
	$$
	The first $\ell^2$-Betti number is given by $\beta_1^{(2)}(G)
	= \dim_{\mathcal{N}(G)}(\ker \partial_1)
	- \dim_{\mathcal{N}(G)}(\overline{\mathrm{Im}(\partial_2)})$.
	Using the additivity of the von Neumann dimension in weakly exact sequences
	(cf.\ \cite[Theorem~1.12(2)]{Luck}), we have
	$$
	\dim_{\mathcal{N}(G)}(\ker \partial_1)
	= \dim_{\mathcal{N}(G)}(\ell^2(G)^d)
	- \dim_{\mathcal{N}(G)}(\overline{\mathrm{Im}(\partial_1)}).
	$$
	Substituting, we obtain
	$$
	\beta_1^{(2)}(G)
	= \dim_{\mathcal{N}(G)}(\ell^2(G)^d)
	- \dim_{\mathcal{N}(G)}(\overline{\mathrm{Im}(\partial_1)})
	- \dim_{\mathcal{N}(G)}(\overline{\mathrm{Im}(\partial_2)}).
	$$
	
	Since $G$ is infinite, the zeroth $\ell^2$-Betti number vanishes by Lemma~\ref{LemmaZeroth}, so the complex is exact at $\ell^2(G)$. Thus
	$$
	\dim_{\mathcal{N}(G)}(\overline{\mathrm{Im}(\partial_1)})
	= \dim_{\mathcal{N}(G)}(\ell^2(G))
	= 1.
	$$
	Finally, $\beta_1^{(2)}(G)
	= d - 1 - \dim_{\mathcal{N}(G)}(\overline{\mathrm{Im}(\partial_2)})
	\le d-1$, as required.
\end{proof}

We end this section by recording the known values of $\ell^2$-Betti numbers for the arithmetic groups $\Aut(\mathbb{Z}^n)=\mathrm{GL}_n(\mathbb{Z})$, which will play a key role later in the paper.
\begin{theorem}\label{BettiGL(Z)}
	If $k=1$ then $\mathrm{GL}_1(\mathbb{Z})=\mathbb{Z}/2\mathbb{Z}$ so $\beta_0^{(2)}(\mathrm{GL}_1(\mathbb{Z}))=\frac{1}{2}$ and $\beta_{n}^{(2)}(G)=0$ for all $n\geq 1$. If $k = 2$ then $\beta_1^{(2)}(\mathrm{GL}_2(\mathbb{Z})) = \frac{1}{24}$ and $\beta_n^{(2)}(\mathrm{GL}_2(\mathbb{Z})) = 0$ for all $n \neq 1$. If $k\geq 3$ then all $\ell^2$-Betti numbers of $\mathrm{GL}_k(\mathbb{Z})$ vanish.
\end{theorem}
\begin{proof}
For $k=1$ the result follows from Lemma \ref{LemmaZeroth} For $k=2$ it is known that $F_2$ embeds into $\mathrm{GL}_2(\mathbb{Z})$ as a subgroup of index $24$ (c.f. \cite{Newmann}). Since $\beta_1^{(2)}(F_2)=1$ and  $\beta_n^{(2)}(F_2)=0$ for all $n\neq1$ the result follows from  Proposition \ref{FiniteIndexBetti}. For $k\geq 3$ the result follows from the fact that all latices in a given locally compact group are measure equivalent, that the vanishing of $\ell^2$-Betti numbers is invariant under measured equivalence (c.f. \cite{Gaboriau 2}) and that $\beta_n^{(2)}(\mathrm{GL}_k(\mathbb{R}))=0$ for all $n\geq 0$  (c.f. \cite{Borel} and \cite{Luck} Theorem 7.34).
\end{proof}
\subsection{Algebraic fibring}
A group $G$ is said to \textbf{fibre algebraically} if there exists a non-trivial epimorphism $\varphi: G\to\mathbb{Z}$ with finitely generated kernel. By Proposition \ref{VanishingBetti} it follows that if $G$ virtually fibres algebraically then $\beta_1^{(2)}(G)=0$. 

We now state and prove some easy properties that will be needed later.
\begin{lemma}\label{FibringFiniteIndex}
Let $G$ be a group and $H\leq G$ be a finite index subgroup. Then, $G$ is virtually fibred if and only if $H$ is virtually fibred.
\end{lemma}
\begin{proof}
	If $H$ is virtually algebraically fibred, then so is $G$ trivially. We may thus assume that $G$ is virtually fibred. Then, there exists a finite index subgroup $G_1\leq G$ and an epimorphism $\varphi:G_1\to\mathbb{Z}$ with finitely generated kernel. Set $H_1=G_1\cap H$, which is a finite index subgroup of both $G_1$ and $H$. Let $\varphi_1:H_1\to\mathbb{Z}$ be the restriction of $\varphi$ to $H_1$. Since $H_1$ is of finite index in $G_1$, the image $\varphi(H_1)$ is a finite index subgroup of $\mathbb{Z}$, hence it is isomorphic to $n\mathbb{Z}$ for some integer $n\geq 1$. Define $\varphi_2:H_1\to\mathbb{Z}$ as the composition of $\varphi_1$ with the isomorphism $n\mathbb{Z}\to\mathbb{Z}$ given by $x\mapsto\frac{x}{n}$. Then, $\varphi_2$ is an epimorphism such that $\ker(\varphi_2)=\ker(\varphi)\cap H_1$. Hence, $\ker(\varphi_2)$ has finite index in $\ker(\varphi)$, so $\ker(\varphi_2)$ is finitely generated. Thus, $H_1$ fibres algebraically, so $H$ virtually fibres.
\end{proof}
\begin{lemma}\label{FibringComposition}
	Let $G$ and $H$ be two groups. Suppose that $G$ is algebraically fibred and that there exists an epimorphism $\varphi:H\to G$ with finitely generated kernel. Then, $H$ is also algebraically fibred.
\end{lemma}
\begin{proof}
	Since $G$ is algebraically fibred there exists an epimorphism $\psi:G\to\mathbb{Z}$ whose kernel is finitely generated. Let us denote the generators of $\ker(\psi)$ by $\lbrace g_1,\dots,g_n\rbrace$. For each $i=1,\dots,n$ choose $\overline{g_i}\in H$ such that $\varphi(\overline{g_i})=g_i$. By hypothesis, $\ker(\varphi)$ is finitely generated, say by $\lbrace h_1,\dots,h_r\rbrace$. Consider the composition $f=\psi\circ\varphi:H\to\mathbb{Z}$, then $\ker(f)$ is generated by the set $\lbrace\overline{g_1},\dots,\overline{g_n},h_1,\dots,h_r\rbrace$. 
\end{proof}

\begin{corollary}\label{VFibringComposition}
	Let $G$ and $H$ be two groups such that $G$ virtually fibres. If there exists an epimorphism $\varphi: H\to G$ with finitely generated kernel, then $H$ is also virtually fibred.
\end{corollary}
\begin{proof}
	Let $G_1$ be a finite index subgroup of $G$ that fibres, with fibration $f: G_1\to\mathbb{Z}$. Consider the preimage $H_1=\varphi^{-1}(G_1)$, which is a finite index subgroup of $H$. Let $\varphi_1:H_1\to G_1$ be the restriction of $\varphi$ to $H_1$. Then, $\ker(\varphi_1)=\ker(\varphi)\cap H_1$, which is a finite index subgroup of the finitely generated group $\ker(\varphi)$. Hence, $\ker(\varphi_1)$ is finitely generated. Now, since $G_1$ fibres and there is an epimorphism $\varphi_1:H_1\to G_1$ with finitely generated kernel, it follows from Lemma \ref{FibringComposition} that $H_1$ is algebraically fibred. Therefore, $H$ virtually fibres.
\end{proof}
\subsection{Right-angled Artin groups}
Let us recall some results and definitions about right-angled Artin groups that will be used in the paper.

\begin{definition} Let $\Gamma$ be a simplicial graph and $\Delta\subset\Gamma$ a subgraph. We say that $\Delta$ is an \textbf{induced subgraph} if every $v,w\in V(\Delta)$ which are connected in $\Gamma$ are also connected in $\Delta$. A subgroup $A_\Delta$ of $A_\Gamma$ generated by an induced subgraph  $\Delta\subset\Gamma$ is said to be a \textbf{special subgroup} of $A_\Delta$. If $v\in V(\Gamma)$, the \textbf{link} $\mathrm{lk}(v)$ is the induced subgraph of $\Gamma$ spanned by the vertices adjacent to $v$ in $\Gamma$. The \textbf{star} $\mathrm{st}(v)$ is the induced subgraph of $\Gamma$ spanned by the vertices of $\mathrm{lk}(v)\cup\lbrace v\rbrace$.
\end{definition}
\begin{proposition}[\cite{Charney 4} Proposition 2.2]\label{CentreRAAG}
	Let $A_\Gamma$ be a RAAG. Then, the centre $\mathrm{Z}(A_\Gamma)$ of $A_\Gamma$ is isomorphic to the RAAG $A_\Delta$ where $\Delta$ is the induced subgraph of $\Gamma$ spanned by the vertices $v\in V(\Gamma)$ such that $\st(v)=\Gamma$.
\end{proposition}
Free products and direct products of RAAGs are easy to understand in terms of the defining graph $\Gamma$: $A_\Gamma$ splits as a non trivial direct product $A_{\Gamma_1}\times A_{\Gamma_2}$ if and only if $\Gamma$ decomposes as a join $\Gamma_1\ast\Gamma_2$ and  $A_\Gamma$ splits as a non trivial free product $A_{\Gamma_1}\ast A_{\Gamma_2}$ if and only if $\Gamma$ decomposes as a disjoint union $\Gamma_1\sqcup\Gamma_2$. Both facts follows from the isomorphism theorem for RAAGs, i.e. $A_\Gamma\cong A_\Theta$ if and only if $\Gamma\cong\Theta$. A proof of this result may be found in \cite{Laurence}.

 The $\ell^2$-Betti numbers of a RAAG are well-known:
\begin{theorem}[Davis-Leary \cite{Davis}]\label{BettiRAAG} Let $A_\Gamma$ be a RAAG. Then:
	$$\beta_{i+1}^{(2)}(A_\Gamma)=\overline{\beta}_i(\widehat{\Gamma})$$
	Here $\overline{\beta}_i$ represents the reduced $i$-th Betti number and $\widehat{\Gamma}$ is the flag complex constructed by adding an $(n-1)$-simplex for each $n$-clique of $\Gamma$.
\end{theorem}
\begin{remark}
	This result was recently generalized by Avramidi-Okun-Schreve in \cite{Avramidi} for the $\ell^2$-Betti numbers of a RAAG with coefficients in an arbitrary field $\mathbb{F}$.
\end{remark}
The \textbf{Bestvina-Brady group} $BB_\Gamma$ associated to a RAAG $A_\Gamma$ is the kernel of the map $\varphi: A_\Gamma\to\mathbb{Z}$ defined as $\varphi(v)=1~\forall~v\in V(\Gamma)$. The homological finiteness properties of those groups are well-known:
\begin{theorem}[Bestvina-Brady \cite{Bestvina}]\label{Bestvina-Brady} Let $A_\Gamma$ be a RAAG. Then $BB_\Gamma$ is $FP_{n+1}$ if and only if $\widehat{\Gamma}$ is $n$-acyclic and $BB_\Gamma$ is $FP$ if and only if $\widehat{\Gamma}$ is acyclic.
\end{theorem}
\begin{corollary}\label{FibringRAAGs}
	Let $A_\Gamma$ be a RAAG. Then, $A_\Gamma$ virtually fibres if and only if $\Gamma$ is connected
\end{corollary}
\begin{proof}
	If $A_\Gamma$ is virtually fibred then $\beta_1^{(2)}(A_\Gamma)=0$ so, by Theorem \ref{BettiRAAG}, $\overline{\beta}_0(\widehat{\Gamma})=0$ and $\Gamma$ is connected. If $\Gamma$ is connected then $BB_\Gamma$ is finitely generated, so the map $\varphi:A_\Gamma\to\mathbb{Z}$ given by sending all standard generators to one is a fibration of $A_\Gamma$. Conversely, if $\Gamma$ is disconnected  $\beta_1^{(2)}(A_\Gamma)=\overline{\beta}_0(\widehat{\Gamma})\neq 0$ so $A_\Gamma$ is not virtually fibred.
\end{proof}
\subsection{Automorphism group of a RAAG}\label{OUTRAAG}
The automorphism group $\Aut(A_\Gamma)$ of a RAAG $A_\Gamma$ has a canonical finite generating set. To describe it we need to introduce a standard partial order on the set of vertices of $\Gamma$. This is defined as $u\leq v$ if $\mathrm{lk}(u)\subset\mathrm{st}(v)$. We can consider an equivalence relation of the vertex set by saying $v\sim w$ if and only if $v\leq w$ and $w\leq v$ and equivalence classes can be ordered, i.e. $[v_1]\leq[v_2]$ if and only if $w_1\leq w_2$ for some $w_1\in[v_1],w_2\in[v_2]$.
\begin{theorem}[Laurence \cite{Laurence}] $\Aut(A_\Gamma)$ is generated by the following automorphisms:
	\begin{itemize}
		\item \textbf{Graph symmetries}: any $\varphi\in\Aut(\Gamma)$ induces an automorphism $\overline{\varphi}\in\Aut(A_\Gamma)$.
		\item \textbf{Inversions}: for each $v\in V(\Gamma)$ there is a map $\iota_v\in\Aut(A_\Gamma)$ given by $\iota_v(v)=v^{-1}$ and $\iota_v(w)=w$ for every $w\in V(\Gamma)\setminus\lbrace v\rbrace$.
		\item \textbf{Transvections}: if $v,w\in V(\Gamma)$ the map $\tau^v_w$ given by $\tau^v_w(w)=wv$ and $\tau^v_w(u)=u$ for every $u\in V(\Gamma)\setminus\lbrace w\rbrace$ is a well defined element of $\Aut(A_\Gamma)$ if and only if $w\leq v$.
		\item \textbf{Partial conjugations}: if $v\in V(\Gamma)$ is such that  $\Gamma\setminus\mathrm{st}(v)\neq\emptyset$ and $C$ is a connected component of $\Gamma\setminus\mathrm{st}(v)$  there is a map $\pi_C^v\in\Aut(A_\Gamma)$ given by $\pi_C^v(w)=vwv^{-1}$ if $w\in V(C)$ and $\pi_C^v(w)=w$ if $w\in V(\Gamma\setminus C)$.
	\end{itemize}
\end{theorem}
In $\Out(A_\Gamma)=\Aut(A_\Gamma)/\mathrm{Inn}(A_\Gamma)$ all generators are non-trivial except for the inner automorphisms.
	
Day proved in \cite{Day 3} that $\Aut(A_\Gamma)$ is not only finitely generated but finitely presented. We define $\Aut^0(A_\Gamma)$ (respectively, $\Out^0(A_\Gamma)$) to be the subgroup of $\Aut(A_\Gamma)$ (respectively, $\Out(A_\Gamma)$) generated by inversions, transvections and partial conjugations. Similarly, we define $\mathrm{SAut}^0(A_\Gamma)$ (respectively, $\mathrm{SOut}^0(A_\Gamma)$) as the subgroup generated by transvections and partial conjugations. $\Aut^0(A_\Gamma)$ and $\mathrm{SAut}^0(A_\Gamma)$ (respectively, $\Out^0(A_\Gamma)$ and $\mathrm{SOut}^0(A_\Gamma)$) are subgroups of finite index of $\Aut(A_\Gamma)$ (respectively, $\Out(A_\Gamma)$), see \cite{Wade} Proposition 3.6.

Let us state some definitions and results about automorphism groups of RAAGs that will be used later.
\begin{definition}
	Let $\Gamma$ be a simplicial graph and $u,v\in V(\Gamma)$ be two vertices that are not adjacent in $\Gamma$. We say that the pair $(u,v)$ forms a \textbf{separating intersection of links} (SIL) if there is a connected component of $\Gamma\setminus(\lk(u)\cap\lk(v))$ that does not contain $u$ nor $v$.
\end{definition}
The following lemma clarifies why the absence of SILs is a condition that is meaningful only for RAAGs with underlying connected graph.
\begin{lemma}\label{NoSILs}
	If $\Gamma$ is a disconnected simplicial graph with no SILs then $A_\Gamma\cong\mathbb{Z}^n\ast \mathbb{Z}^m$ for some $n,m\geq1$.
\end{lemma}
\begin{proof}
	If the graph $\Gamma$ has at least $3$ connected components $\Gamma_1,\Gamma_2$ and $\Gamma_3$ then for any two vertices $u\in V(\Gamma_1)$ and $v\in V(\Gamma_2)$, the pair $(u,v)$ forms a SIL. Indeed, since $\Gamma_3$ is a connected component of $\Gamma\setminus(\lk(u)\cap\lk(v))$, it witnesses the SIL condition. Hence, we may assume that $\Gamma=\Gamma_1\sqcup\Gamma_2$ with $\Gamma_1$ and $\Gamma_2$ connected graphs. 
	
	If $\Gamma_1$ is not a complete graph there exists two vertices $u,v\in V(\Gamma_1)$ that are not connected in $\Gamma$. Then, $\Gamma_2$ is a connected component of  $\Gamma\setminus(\lk(u)\cap\lk(v))$ and so the pair $(u,v)$ forms a SIL. It follows that $\Gamma_1$ is a complete graph and by the same reasoning $\Gamma_2$ is a complete graph and the claim follows.
\end{proof}
Every transvection and partial conjugation defines an automorphism of infinite order, while graph symmetries and inversions are of finite order. Therefore, one can easily check that $\Out(A_\Gamma)$ is finite if and only if it contains no transvections and no partial conjugations (see \cite{Wade} Proposition 3.6).
\begin{lemma}\label{FiniteOutLemma}
	Let $\Gamma$ be a simplicial graph with more than one vertex. If $\Out(A_\Gamma)$ is finite then $\Gamma$ is connected and the centre of $A_\Gamma$ is trivial. In particular, $A_\Gamma$ has finite index in $\Aut(A_\Gamma)$.
\end{lemma}
\begin{proof}
	If $A_\Gamma=\mathbb{Z}^n$ then  $\Out(A_\Gamma)=\mathrm{GL}_n(\mathbb{Z})$ is an infinite group. Hence, we may assume that  $\mathrm{Z}(A_\Gamma)\neq A_\Gamma$. If $\mathrm{Z}(A_\Gamma)\neq1$ then $A_\Gamma=A_{\Gamma_1}\times A_{\Gamma_2}$, where $\Gamma_1,\Gamma_2\neq\emptyset$ and $\Gamma_1$ is complete. Consider $u\in V(\Gamma_2)$ and $v\in V(\Gamma_1)$, since $\st(v)=\Gamma$ we have that $\tau_u^v$ is a well-defined element of infinite order in $\Out(A_\Gamma)$. Hence $\Out(A_\Gamma)$ is infinite.
	
	 Assume that $\Gamma$ is disconnected. Since $\Out(A_\Gamma)$ is finite we must have that $\Gamma\setminus\st(v)$ is connected for every $v\in V(\Gamma)$, otherwise there must be a non-inner partial conjugation with respect to the vertex $v$. This forces $\Gamma$ to have at most two connected components, otherwise $\Gamma\setminus\st(v)$ would be disconnected for every $v\in V(\Gamma)$. Let  $\Gamma=\Gamma_1\sqcup\Gamma_2$  be the decomposition into connected components and assume $\Gamma_i$ is not a complete graph for $i\in\lbrace 1,2\rbrace$ . Take $v\in V(\Gamma_i)$ such that $\st(v)\neq \Gamma_i$, then $\Gamma\setminus\st(v)$ is disconnected. This implies that both $\Gamma_1$ and $\Gamma_2$ are complete graphs. If $V(\Gamma_i)>1$ for $i\in\lbrace 1,2\rbrace$ the transvection $\tau_u^v$ is a well-defined element of infinite order in $\Out(A_\Gamma)$ for every $u,v\in V(\Gamma_i)$. Hence $V(\Gamma_1)=V(\Gamma_2)=1$ and $A_\Gamma$ is the free group in two generators, which has infinite outer automorphism group.
	 
	 For the second part, note that the centre of $A_\Gamma$ is trivial. Hence, the inner automorphism group $\mathrm{Inn}(A_\Gamma)$, which has finite index in $\Aut(A_\Gamma)$, is isomorphic to $A_\Gamma$.
\end{proof}
To conclude this section, we present several results that will later serve as tools for constructing examples of RAAGs exhibiting distinct $\ell^2$-Betti numbers and different algebraic fibring properties.
\begin{proposition}\label{nsphereRAAG}
	For each $n\geq 1$ there exists a simplicial graph $\Gamma_n$ such that $\widehat{\Gamma_n}\cong\mathbb{S}^n$, where $\mathbb{S}^n$ denotes the $n$-sphere, and such that $\Out(A_{\Gamma_n})$ is finite.
\end{proposition}

\begin{proof}
	Let $\lbrace e_i\rbrace_{i=1}^{n+1}$ be the canonical basis of $\mathbb{R}^{n+1}$. Define $\Delta_n$ to be the graph with vertex set $V_1=\lbrace \pm e_i\rbrace_{i=1}^{n+1}$ and where two vertices are adjacent unless they are exact opposites. This graph $\Delta_n$ is the $1$-skeleton of the $n$-octahedron, so its flag complex $\widehat{\Delta_n}$ is homotopy equivalent to $\mathbb{S}^{n}$. Now, define $\Gamma_n$ to be the $1$-skeleton of the barycentric subdivision of $\widehat{\Delta_n}$. Since barycentric subdivision preserves the homotopy type, we see that $\widehat{\Gamma_n}\simeq \mathbb{S}^n$. For each tuple of signs $a\in\lbrace \pm 1\rbrace^{n+1}$, define the vector $f_a=\frac{1}{n+1}a\in\mathbb{R}^{n+1}$. Define $V_2=\lbrace f_a\mid a\in \lbrace \pm 1\rbrace^{n+1}\rbrace$. The vertex set of $\Gamma_n$ is equal to $V_1\cup V_2$ and the edge set is as follows: two vertices of $V_1$  are adjacent unless they are negatives of each other, a vertex $\pm e_i\in V_1$ is adjacent to $f_a\in V_2$ if and only if the sign of the $i$-th component of $f_a$ matches that of $\pm e_i$ and no vertices of $V_2$ are adjacent. From this construction of $\Gamma_n$ it is straightforward to check that $\Gamma_n\setminus\st(v)$ is connected for every $v\in V(\Gamma_n)$ and if $\lk(v)\subset\st(w)$ for some $v,w\in V(\Gamma_n)$ then $v=w$. Hence, $\Out(A_\Gamma)$ is finite and the claim follows.
\end{proof}
The following result shows that it is possible to construct RAAGs whose outer automorphism groups share the same virtual properties as those of other RAAGs.
\begin{theorem}[Wiedmer \cite{Wiedmer}]\label{FiniteindexRAAG} For any RAAG $A_\Gamma$ there exists a RAAG $A_\Lambda$ such that $A_\Gamma$ is a finite index subgroup of $\Out(A_\Lambda)$.
\end{theorem}
The following result shows that the virtual properties of the (outer) automorphism group of certain direct products can be deduced from those of the direct product of the (outer) automorphism groups of the individual factors. First, recall that a group $G$ is said to be \textbf{directly indecomposable} if there are no subgroups $G_1,G_2\leq G$ such that $G=G_1\times G_2$.
\begin{proposition}[\cite{Qiang} Theorem 2.5]\label{CenterAutomorphism}
	Let $G_1,\dots, G_n$ be centreless directly indecomposable groups and $k_1,\dots,k_n$ positive integers. Then:
	\begin{enumerate}
		\item $\Aut(G_1^{k_1}\times\cdots\times G_n^{k_n})=\prod\limits_{i=1}^n\left(\prod\limits_{k_i}\Aut(G_i^{})\right)\rtimes S_{k_i}$
		\item $\Out(G_1^{k_1}\times\cdots\times G_n^{k_n})=\prod\limits_{i=1}^n\left(\prod\limits_{k_i}\Out(G_i^{})\right)\rtimes S_{k_i}$
	\end{enumerate}
\end{proposition}
\subsection{Pure symmetric automorphism group}
The \textbf{pure symmetric automorphism group} $\mathrm{PSA}(A_\Gamma)$ of a right-angled Artin group (RAAG) $A_\Gamma$ is the subgroup of $\Aut(A_\Gamma)$ generated by all partial conjugations, while the \textbf{pure symmetric outer automorphism group} $\mathrm{PSO}(A_\Gamma)$ is its image in $\Out(A_\Gamma)$. It is easily verified that $\mathrm{PSA}(A_\Gamma)$ (respectively, $\mathrm{PSO}(A_\Gamma)$) has finite index in $\Aut(A_\Gamma)$ (respectively, $\Out(A_\Gamma)$) if and only if $A_\Gamma$ admits no transvections. In such cases, we say that $A_\Gamma$ is \textbf{transvection-free}.

\begin{theorem}[\cite{Toinet}, Theorem 3.1, \cite{Koban} Theorem 3.3]\label{Presentation PSA}
	Let $A_\Gamma$ be a RAAG. The group $\mathrm{PSA}(A_\Gamma)$ admits a finite presentation with generators given by the set of all partial conjugations, and with relations consisting entirely of commutators among these generators. As a consequence, the group $\mathrm{PSO}(A_\Gamma)$ admits a finite presentation obtained by adding, for each $v\in V(\Gamma)$, the relator $\prod_{C\in I_v}\pi^a_C=1$ where $I_v$ denotes the set of connected components of $\Gamma\setminus\st(v)$ and the product is taken over all partial conjugations associated to $v$.
\end{theorem}
These explicit presentations can be found in \cite{Koban}. Before moving to the $\Out(A_\Gamma)$ case, let us record a simple observation about $\mathrm{PSA}(A_\Gamma)$.

\begin{lemma}\label{PSA-NO-PC}
	Assume that $\Gamma$ is a simplicial graph such that $\Gamma\setminus\st(v)$ is connected for every $v\in V(\Gamma)$. Then $\Gamma$ has no SILs, and $\mathrm{PSA}(A_\Gamma)$ is isomorphic to $A_{\Gamma\setminus\Delta}$, where $\Delta$ is the induced subgraph of $\Gamma$ spanned by those vertices $v\in V(\Gamma)$ with $\st(v)=\Gamma$.
\end{lemma}

\begin{proof}
	For any pair of vertices $u,v\in V(\Gamma)$ we have $\Gamma\setminus\st(v)\subseteq \Gamma\setminus(\lk(u)\cap\lk(v))$.
	Since $\Gamma\setminus\st(v)$ is connected by assumption, it follows that $\Gamma\setminus(\lk(u)\cap\lk(v))$ is also connected. Hence $\Gamma$ has no SILs.
	
	On the other hand, every partial conjugation is an inner automorphism. Therefore, $\mathrm{PSA}(A_\Gamma)=\mathrm{Inn}(A_\Gamma)=A_\Gamma / Z(A_\Gamma)$, and the result follows from Proposition~\ref{CentreRAAG}.
\end{proof}

For the case of $\Out(A_\Gamma)$, Day and Wade studied the conditions under which $\mathrm{PSO}(A_\Gamma)$ is isomorphic to a RAAG. To state their result, we first introduce the following definition.
\begin{definition}
	Let $\Gamma$ be a simplicial graph and $v\in V(\Gamma)$. The  \textbf{support graph} $\Delta_v$ is defined as follows. The vertex set of $\Delta_v$ is in bijection with the set of connected components of $\Gamma\setminus \st(v)$. Two components $K$ and $L$ are joined by an edge if there exists a vertex $w\in V(K)$ such that $L$ is a connected component of both $\Gamma\setminus\st(w)$ and $\Gamma\setminus\st(v)$.  
\end{definition}
\begin{theorem}[Day-Wade, \cite{Day 4}]\label{PSO-RAAG} Let $A_\Gamma$ be a RAAG. Then, $\mathrm{PSO}(A_\Gamma)$ is isomorphic to a RAAG $A_\Theta$ if and only if $\Delta_v$ is a forest for every $v\in V(\Gamma)$. In that case, $\Theta$ is constructed as follows. For each $v\in V(\Gamma)$ choose a distinguished component $C_v$ of $\Delta_v$. The vertex set of $\Theta$ consists of:
\begin{enumerate}
	\item[1)] vertices of the form $v_e^u$ for each vertex $u\in V(\Gamma)$ and $e\in E(\Delta_ u)$ 
	\item[2)] vertices of the form  $v_C^u$ for each vertex $u\in V(\Gamma)$ and $C$ a connected component of $\Delta_v$ distinct from $C_v$.
\end{enumerate}
The graph $\Theta$ is given by the following edges:
\begin{itemize}
	\item The vertices of type $2)$ are connected to every other vertex of $\Theta$.
	\item  Two vertices $v_{e}^{u}$ and $v_{f}^{w}$ of type $1)$ are connected in $\Theta$ unless the following conditions are satisfied: $(u,w)$ forms a SIL, $e=\lbrace L,M\rbrace,f=\lbrace L,N\rbrace$, $w\in V(M),u\in V(N)$ and $L$ is a common connected component of $\Gamma\setminus\st (u)$ and $\Gamma\setminus\st(v)$.
\end{itemize}
\end{theorem}
Moreover, using this result, Wiedmer showed in \cite{Wiedmer} that every RAAG arises  as the pure symmetric outer automorphism of a RAAG associated to a larger graph. In other words, for any RAAG $A_\Gamma$ there exists a simplicial graph $\Theta$ such tat $A_\Gamma\cong\mathrm{PSO}(A_\Theta)$. The construction of the graph $\Theta$ is explicit and depends only on $\Gamma$. Furthermore, if $\Gamma$ has the property that $\Gamma\setminus\st(v)$ has at most $2$ connected components for every $v\in V(\Gamma)$ then, the result of Day and Wade applies directly. Indeed, in this case, each support graph has at most $2$ vertices, and any graph with at most $2$ vertices is necessarily a forest.
\section{Algebraic fibring}

\subsection{Pure symmetric automorphism group}\label{SymetricAUT}
Given a RAAG $A_\Gamma$, the Bieri–Neumann–Strebel invariants (BNS-Invariants) $\Sigma^1(\mathrm{PSA}(A_\Gamma))$ and $\Sigma^1(\mathrm{PSO}(A_\Gamma))$ provide powerful tools for detecting algebraic fibring properties of these groups. These invariants were computed in full generality by Koban and Piggott \cite{Koban} for $\mathrm{PSA}(A_\Gamma)$ and by Day and Wade \cite{Day 4} for $\mathrm{PSO}(A_\Gamma)$. In what follows, we briefly recall the definition of the BNSR invariants and use them to investigate the virtual algebraic fibring of pure symmetric (outer) automorphism groups of RAAGs.

\begin{definition}
	Let $G$ be a finitely generated group and $S$ a generating set of $G$. The \textbf{character sphere} of $G$ is defined as: 
	$$S(G)=\left(\mathrm{Hom}(G,\mathbb{R})\setminus\lbrace0\rbrace\right)/\sim$$
	where $\chi_1\sim\chi_2$ iff there exists $t>0$ such that $\chi_1=t\chi_2$. The \textbf{BNS invariant} of $G$ is defined as:
	$$\Sigma^1(G)=\lbrace[\chi]\in S(G)\mid\mathrm{Cay}_\chi(G,S)\text{ is connected}\rbrace$$
	Here, $\mathrm{Cay}(G,S)$ is the Cayley graph of $G$ with respect to the generating set $S$ and    $\mathrm{Cay}_\chi(G,S)$ is the subgraph of the Cayley graph induced by the vertices $v$ with $\chi(v)\geq 0$.
\end{definition}
The invariant is well defined, i.e. it does not depend on the choice of generating set $S$ [c.f. \cite{Strebel} Theorem A2.3]. The reason why $\Sigma^1(G)$ encodes the fibring properties of $G$ is the following central theorem:
\begin{theorem}[Corollary A4.3 \cite{Strebel}]\label{Sigma1Fibring} Let $G$ be finitely generated and $\chi:G\to\mathbb{Z}$. Then, $\ker(\chi)$ is finitely generated if and only if $[\chi],[-\chi]\in\Sigma^1(G)$.
\end{theorem}
Characters of the form $\chi:G\to\mathbb{Z}$ are known as \textbf{discrete characters}. The invariant $\Sigma^1(G)$ is \textbf{symmetric} if $[\chi]\in\Sigma^1(G)$ if and only if $[-\chi]\in\Sigma^1(G)$. 
\begin{corollary}\label{EmptySigma1} Let $G$ be a finitely generated group such that $\Sigma^1(G)$ is symmetric. Then, $G$ fibres if and only if $\Sigma^1(G)\neq\emptyset$.
\end{corollary}
\begin{proof}
	If $G$ fibres, then by Theorem \ref{Sigma1Fibring}, we have $\Sigma^1(G)\neq\emptyset$ . Conversely, suppose $\Sigma^1(G)\neq\emptyset$. Then, there exists a character $\chi:G\to\mathbb{R}$ such that $[\chi]\in\Sigma^1(G)$. Since the set of discrete characters is dense in the character sphere $S(G)$ (cf. \cite{Strebel} Lemma B3.24) and $\Sigma^1(G)$ is an open subset of $S(G)$ (cf. \cite{Bieri-Renz}) it follows that there is a  discrete character $\overline{\chi}:G\to\mathbb{Z}$ such that  $[\overline{\chi}]\in\Sigma^1(G)$. Moreover, as $\Sigma^1(G)$ is symmetric, we also have $[\overline{\chi}],[-\overline{\chi}]\in \Sigma^1(G)$. Hence, by Theorem \ref{Sigma1Fibring}, it follows that $G$ fibres.
\end{proof}
To define $\Sigma^1(\mathrm{PSA}(A_\Gamma))$ and $\Sigma^1(\mathrm{PSO}(A_\Gamma))$ we need to introduce a previous notion.
\begin{definition}
	Let $A_\Gamma$ be a RAAG and $X$ the set of all partial conjugations. A subset $S\subset X$ is a \textbf{p-set} if:
	\begin{itemize}
		\item For each $v\in V(\Gamma)$ there is at most one partial conjugation $\pi^v_K\in S$
		\item $S$ has a non-trivial partition $S=S_1\cup S_2$ such that for every $\pi^v_K\in S_1$ and $\pi^w_L\in S_2$ we have $v\in V(L)$ and $w\in V(K)$.
	\end{itemize}
	A subset $S\subset X$ is a \textbf{$\delta$-p-set} if:
	\begin{itemize}
		\item For each $v\in V(\Gamma)$ there are exactly two or zero partial conjugations of the form $\pi^v_K$ in $S$.
		\item $S$ has a non-trivial partition $S=S_1\cup S_2$ such that for every $\pi^v_K\in S_1$ and $\pi^w_L\in S_2$ we have $v\in V(L)$ or $w\in V(K)$ or $L=K$.
	\end{itemize}
\end{definition}
\begin{theorem}[Koban-Piggott \cite{Koban}, Day-Wade \cite{Day 4}]\label{SigmaPSO}
	Let $A_\Gamma$ be a RAAG and $\chi_1:\mathrm{PSA}(A_\Gamma)\to\mathbb{R},\chi_2:\mathrm{PSO}(A_\Gamma)\to\mathbb{R}$ two non-zero characters. Then:
	\begin{enumerate}
		\item $[\chi_1]\notin\Sigma^1(\mathrm{PSA}(A_\Gamma))$ if and only if one of the following holds:
		\begin{itemize}
			\item $\chi_1$ is non-trivial on some inner automorphism and the support of $\chi_1$ is a subset of a p-set.
			\item $\chi_1$ is trivial on every inner automorphism and the support of $\chi_1$ is a subset of a $\delta$-p-set.
		\end{itemize}
		\item $[\chi_2]\notin\Sigma^1(\mathrm{PSO}(A_\Gamma))$ if and only if  the support of $\chi_2$ is a subset of a $\delta$-p-set.
\end{enumerate}
\end{theorem}

As a consequence of Theorem~\ref{SigmaPSO}, both $\mathrm{PSA}(A_\Gamma)$ and $\mathrm{PSO}(A_\Gamma)$ have symmetric $\Sigma^1$-invariants. In particular, their algebraic fibring depends solely on whether $\Sigma^1(\mathrm{PSA}(A_\Gamma))$ and $\Sigma^1(\mathrm{PSO}(A_\Gamma))$ are non-empty.

\begin{proposition}\label{FibringPSA} Let $A_\Gamma$ be a RAAG. Then, $\mathrm{PSA}(A_\Gamma)$ fibres if and only if virtually fibres and this happens if and only if $A_\Gamma\ncong \mathbb{Z}^n,\mathbb{Z}^n\ast\mathbb{Z}^m$ for any $n,m\geq 1$.
\end{proposition}
\begin{proof}
	 If $A_\Gamma=\mathbb{Z}^n$ for some $n\geq 1$ then $\mathrm{PSA}(A_\Gamma)$ is trivial, so we may assume that $A_\Gamma\ncong \mathbb{Z}^n$ for any $n\geq 1$. Assume that there exists $v\in V(\Gamma)$ such that $\Gamma\setminus\st(v)$ is disconnected and let $L,K$ be two connected components of $\Gamma\setminus\st(v)$. Let $w\in V(\Gamma)\setminus\lbrace v\rbrace$ such that $\Gamma\setminus\st(w)\neq\emptyset$. Such a vertex exists because if $\Gamma\setminus\st(w)=\emptyset$ for every  $w\in V(\Gamma)\setminus\lbrace v\rbrace$ then $\Gamma\setminus\lbrace v\rbrace$ is a complete graph, which would imply that $\Gamma\setminus\st(v)$ is either connected or empty, a contradiction. Define the discrete character $\chi:\mathrm{PSA}(A_\Gamma)\to\mathbb{Z}$ by $\chi(\pi_L^v)=1,\chi(\pi_K^v)=-1,\chi(\pi^w_C)=1$ for every connected component $C$ of $\Gamma\setminus\st(w)$ and $\chi(\pi^u_C)=0$ for any other partial conjugation. By Theorem \ref{Presentation PSA} all the relations of $\mathrm{PSA}(A_\Gamma)$ are preserved under $\chi$, so it is a well-defined character. The character $\chi$ is non-trivial on the inner automorphism induced by conjugation with $w$. Moreover, there are two distinct partial conjugations associated with $v$ with non-zero $\chi$-value, implying that the support of $\chi$ is not the subset of a p-set. Therefore, by Theorem \ref{SigmaPSO}.1, it follows that $[\chi]\in\Sigma^1(\mathrm{PSA}(A_\Gamma))$. Consequently, $\mathrm{PSA}(A_\Gamma)$ fibres by Corollary \ref{EmptySigma1}.

	  Therefore, we may assume that $\Gamma\setminus\st(v)$ is connected for every $v\in V(\Gamma)$. By Corollary \ref{PSA-NO-PC}, $\Gamma$ has no SILs and $\mathrm{PSA}(A_\Gamma)$ is isomorphic to $A_{\Gamma\setminus\Delta}$, where $\Delta$ is the subgraph of $\Gamma$ generated by the $v\in V(\Gamma)$ such that $\mathrm{st}(v)=\Gamma$. By Lemma \ref{NoSILs} either $A_\Gamma$ is connected or $A_\Gamma\cong\mathbb{Z}^n\ast\mathbb{Z}^m$ for some $m,n\geq 1$.  In the latter case, it follows by Theorem 4.15.ii in \cite{Kammeyer} that $\beta_1^{(2)}(\mathbb{Z}^n\ast\mathbb{Z}^m)=1$ and thus $A_\Gamma$ does not virtually fibre. If $\Gamma$ is connected $A_\Gamma$ fibres due to Corollary \ref{FibringRAAGs}.
\end{proof}

\begin{proposition}\label{FibringPSO} Let $A_\Gamma$ be a RAAG. Then, $\mathrm{PSO}(A_\Gamma)$ fibres if and only if virtually fibres and this happens if and only if one of the following two conditions holds:
	\begin{enumerate}
		\item There exists $v\in V(\Gamma)$ such that $\Gamma\setminus\st(v)$ has at least $3$ connected components, or
		\item  $\mathrm{PSO}(A_\Gamma)$ is isomorphic to a RAAG $A_\Theta$ with $\Theta$ a connected graph.
	\end{enumerate}
\end{proposition}
\begin{proof}
	 Assume first that $\Gamma\setminus\st(v)$ has at most $2$ connected components. Then, for each $v\in V(\Gamma)$ each support graph $\Delta_v$ has at most two vertices and is therefore a forest. In this situation we may apply Theorem \ref{PSO-RAAG}, which implies that there exists a simplicial graph $\Theta$ such that $\mathrm{PSO}(A_\Gamma)$ is isomorphic to a RAAG $A_\Theta$. Hence, by Corollary \ref{FibringRAAGs}, $\mathrm{PSO}(A_\Gamma)$ fibres if and only if $\Theta$ is connected.

	 Therefore, we may assume there exists $v\in V(\Gamma)$ such that $\Gamma\setminus\st(v)$ has at least $3$ connected components $L,K$ and $R$. Define a character $\chi:\mathrm{PSO}(A_\Gamma)\to\mathbb{Z}$ by setting: $$\chi(\pi_L^v)=1,\chi(\pi_K^v)=1,\chi(\pi^v_R)=-2$$
	 and $\chi(\pi^u_C)=0$ for any other partial conjugation $\pi^u_C$. The map $\chi$ is well-defined map because all the relators in the standard presentation of  $\mathrm{PSO}(A_\Gamma)$ are commutators except for the relators  of the form  $\prod_{C\in I_w}\pi^w_C=1$ where $w\in V(\Gamma)$ and $I_w$ is the set of connected components of $\Gamma\setminus\st(w)$.
	 These relations are clearly preserved under $\chi$. The support of $\chi$ equals to $\lbrace \pi_L^v,\pi_K^v,\pi^u_C\rbrace$ which consists of $3$ partial conjugations based at the same vertex $v$, but corresponding to three different connected components of $\Gamma\setminus\st(v)$. Therefore, the support of $\chi$ is not the subset of a $\delta$-p-set and, by Theorem \ref{SigmaPSO}.2, $[\chi]\in\Sigma^1(\mathrm{PSO}(A_\Gamma))$. Consequently, $\mathrm{PSO}(A_\Gamma)$ fibres.
\end{proof}

As a corollary, we can determine whether $\Aut(A_\Gamma)$ and $\Out(A_\Gamma)$ virtually fibre for transvection-free RAAGs.
\begin{corollary}\label{FibringTFAut}
	Let $A_\Gamma$ be a transvection-free RAAG such that $\abs{V(\Gamma)}>1$. Then, $\Aut(A_\Gamma)$ virtually fibres.
\end{corollary}
\begin{proof}
	If $A_\Gamma$ is transvection-free then $\mathrm{PSA}(A_\Gamma)$ is a finite-index subgroup of $\Aut(A_\Gamma)$. By Proposition \ref{FibringPSA}, $\mathrm{PSA}(A_\Gamma)$ fibres  if and only if $A_\Gamma\ncong \mathbb{Z}^n,\mathbb{Z}^n\ast\mathbb{Z}^m$ for any $n,m\geq 1$. However, all such groups have transvections, contradicting the assumption that $A_\Gamma$ is transvection-free. Therefore, $\mathrm{PSA}(A_\Gamma)$ fibres, and the claim follows.
\end{proof}
\begin{corollary}
	Let $A_\Gamma$ be a transvection-free RAAG. Then $\Out(A_\Gamma)$ virtually fibres if and only if one of the following holds:
	\begin{enumerate}
		\item There exists $v\in V(\Gamma)$ such that $\Gamma\setminus\st(v)$ has at least $3$ connected components, or
		\item $\mathrm{PSO}(A_\Gamma)$ is isomorphic to a RAAG $A_\Theta$ with $\Theta$ a connected graph.
	\end{enumerate}
\end{corollary}
\begin{proof}
	If either of the two conditions holds, then by Proposition \ref{FibringPSO}, $\mathrm{PSO}(A_\Gamma)$ is a finite-index subgroup of $\Out(A_\Gamma)$ that fibres.
	Otherwise, there exists a graph $\Theta$ such that $\mathrm{PSO}(A_\Gamma)$ is isomorphic to $A_\Theta$. In this case, $\Out(A_\Gamma)$ is virtually a RAAG and, in particular, it is virtually RFRS. By a result of Kielak (cf. \cite{Kielak}) a virtually RFRS group $G$ fibres virtually if and only if $\beta_1^{(2)}(G)=0$. Therefore, $\Out(A_\Gamma)$ virtually fibres if and only if  $\beta_1^{(2)}(A_\Theta)=0$ which, by Theorem \ref{BettiRAAG},  holds if and only if $\Theta$ is connected.
\end{proof}
\subsection{Transvection subgroup of $\Out(A_\Gamma)$}\label{SectionTransvection}
Given a simplicial graph $\Gamma$ with $\abs{\Gamma}=n$ there is a canonical map $\rho:\Out(A_\Gamma)\to\mathrm{GL}_n(\mathbb{Z})$ induced by the action on the abelianization of $A_\Gamma$. The kernel of this map is known as the \textbf{Torelli subgroup} of $\Out(A_\Gamma)$ and is denoted by $\mathrm{IA}_\Gamma$.

Day proved in \cite{Day 2} that $\mathrm{IA}_\Gamma$ is finitely generated by the set of partial conjugations and the set of \textbf{commutator transvections} $[\tau_w^v,\tau_w^u]$ for all $u,v,w\in \Gamma$ with $w\leq u,v$ and $u,v$ not adjacent in $\Gamma$. The same result was independently obtained by Wade in \cite{Wade}. The map $\rho$ induces a short exact sequence:
\begin{equation} \label{eq:1}
	1\to \mathrm{IA}_\Gamma\to \mathrm{SOut}^0(A_\Gamma)\xrightarrow{\rho} Q\to 1
\end{equation}
where $Q$ is a block lower triangular matrix subgroup of $\mathrm{SL}_n(\mathbb{Z})$. The structure of $Q$ encodes information about the transvections of $\Out(A_\Gamma)$. 

To describe $Q$ explicitly, label the vertices of $\Gamma$ as $v_1,\dots,v_n$, ordering them so that vertices of each equivalence class are consecutive and, if $v_i\leq v_j$ are vertices in different equivalence classes, we have $i\leq j$. Let $V_1,\dots,V_N$ denote the equivalence classes, indexed so that if $v\leq w$, then $v\in V_i$ and $w\in V_j$ with $i\leq j$.

We define a directed graph $\Lambda_\Gamma$, called \textbf{transvection graph}, with vertex set $\lbrace V_1,\dots,V_N\rbrace$ and such that there is a directed edge from $V_i$ to $V_j$ if $v\leq w$ for some $v\in V_i$ and $w\in V_j$. There is a loop from $V_i$ to itself if and only if $\abs{V_i}\geq 2$. In this setting, the image of a transvection $\tau_{v_i}^{v_j}$ under $\rho$ is the elementary matrix $\mathrm{E}_{i,j}$ with $1$'s in the diagonal and in the $(i,j)$-entry and zeroes elsewhere. Thus, the matrices in $Q$ are block lower triangular, with diagonal blocks in $\mathrm{SL}_{\abs{V_i}}(\mathbb{Z})$ and non-trivial off-diagonal blocks $(i,j)$ permitted precisely when there is an edge $V_i\to V_j$ in $\Lambda_\Gamma$. 

\begin{theorem}[Proposition 4.11, \cite{Wade}]\label{PresentationQ} The group $Q$ is finitely presented. It is generated by the set: 
	$$\lbrace E_{i,j}\mid \tau_{v_i}^{v_j}\textrm{ is a well-defined transvection}\rbrace$$
with the following defining relations:
\begin{itemize}
	\item $[E_{i,j},E_{k,l}]$ if $i\neq k$ and $j\neq l$.
	\item  $[E_{i,j},E_{j,k}]E_{i,k}^{-1}$ for $i,j,k$ distinct.
	\item $(E_{i,j}E_{j,i}^{-1}E_{i,j})^4$ if $i\neq j$ and $[v_i]=[v_j]$
	\item $E_{i,j}E_{j,i}^{-1}E_{i,j}E_{j,i}E_{i,j}^{-1}E_{j,i}$ if $\lbrace v_i,v_j\rbrace$ is an equivalence class of two vertices.
\end{itemize}
\end{theorem}

There is a condition on the graph $\Gamma$ that will be useful to analyse the fibring properties of the group $Q$.
\begin{definition}
	We say that $\Gamma$ has \textbf{property (A)} if for every pair of distinct vertices $u,v\in\Gamma$ with $u\leq v$ there exists a third vertex $w\neq u,v$ with $u\leq w\leq v$.
\end{definition}
This condition is known in the literature as property (B2) in \cite{Aramayona} and as property (A1) in \cite{Sale}. Property (A) can fail for two distinct reasons:
\begin{enumerate}
	\item[(P1)] There exists an equivalence class of vertices in $\Gamma$ with exactly two elements. If there are exactly $n$ such equivalence classes, we say that \textbf{property (P1.n)} holds.
	\item[(P2)] There exist two singleton equivalence classes $[u]=\lbrace u\rbrace$ and $[v]=\lbrace v\rbrace$  with $u\leq v$, such that there is no vertex $w\neq u,v$ satisfying $u\leq w\leq v$.
\end{enumerate}
Thus, property (A) fails if and only if either (P1) or (P2) holds. In particular, the failure of property (P2) determines how far the group $Q$ is from being perfect.
\begin{lemma}\label{ComutatorQ} Assume that property (P2) fails. Consider two vertices $v_i\leq v_j$ such that $[v_i]\neq\lbrace v_i,v_j\rbrace$. Then, $E_{i,j}\in Q'$.
\end{lemma}
\begin{proof}
	Since property (P2) fails there exists $v_k\neq v_i,v_j$ such that $v_i\leq v_k\leq v_j$. Then, the matrices $E_{i,k}$ and $E_{k,j}$ are generators of $Q$. The commutator relation of the presentation of Theorem \ref{PresentationQ} gives $E_{i,j}=[E_{i,k},E_{k,j}]\in Q'$, as desired.
\end{proof}

\begin{proposition}\label{FibringQ}
	The group $Q$ fibres if and only if property (P2) holds and this is equivalent to $Q^\mathrm{ab}$ being infinite. As a consequence, if property (P2) holds, then $\mathrm{SOut}^0(A_\Gamma)$ fibres and $\Out(A_\Gamma)$ fibres virtually.
\end{proposition}
\begin{proof} If property (P2) holds then there exist $v_i\leq v_j$ such that there is no $v_r\neq v_i,v_j$ with $v_i\leq v_r\leq v_j$. Moreover, $[v_i]=\lbrace v_i\rbrace$ and $[v_j]=\lbrace v_j\rbrace$. Consider the set:
	 $$\mathcal{E}=\lbrace E_{k,l}\mid\tau_{v_k}^{v_l}\text{ is well-defined and }(k,l)\neq(i,j)\rbrace$$
and define the map $f:Q\to\mathbb{Z}$ by $f(E_{i,j})=1$ and $f(E_{k,l})=0$ if $E_{k,l}\in\mathcal{E}$. This is a well-defined group homomorphism because it preserves all relators of $Q$ from Theorem \ref{PresentationQ}. We claim that $\ker(f)=\langle\mathcal{E}\rangle$. As $\ker(f)$ is normally generated by $\mathcal{E}$, it is enough to prove that for every $E_{k,l}\in\mathcal{E}$ there exists $R\in\langle\mathcal{E}\rangle$ such that $E_{i,j}E_{k,l}=RE_{i,j}$. If $i\neq l$ and $j\neq k$ then $E_{i,j}$ and $E_{k,l}$ commute so the claim follows. If $i=l$ and $j\neq k$ then $E_{k,j}\in\mathcal{E}$ and  $E_{i,j}E_{k,i}=E_{k,i}E_{k,j}E_{i,j}$ and if $j=k$ and $i\neq l$ then $E_{i,l}\in\mathcal{E}$ and $E_{i,j}E_{j,l}=E_{j,l}(E_{i,l})^{-1}E_{i,j}$. We deduce that $\ker(f)$ is finitely generated.
	
Now, suppose that property (P2) fails. Consider a pair of distinct vertices $v_i\leq v_j$ then, we have two possibilities. If $[v_i]\neq\lbrace v_i,v_j\rbrace$ we have $E_{i,j}\in Q'$ by Lemma \ref{ComutatorQ}. Therefore, $E_{i,j}=1$ in $Q^\mathrm{ab}$. If $[v_i]=\lbrace v_i,v_j\rbrace$ then, from the presentation of $Q$, we have $(E_{i,j}E_{j,i}^{-1}E_{i,j})^4=1$ and $(E_{j,i}E_{i,j}^{-1}E_{j,i})^4=1$. Those relations imply that $E_{i,j}^{12}=E_{j,i}^{12}=1$ in $Q^{\mathrm{ab}}$, so both $E_{i,j}$ and $E_{j,i}$ have finite order in the abelianization.

For the final part, observe that if property (P2) holds, then $Q$ algebraically fibres. Since $\rho:\mathrm{SOut}^0(A_\Gamma)\to Q$ is an epimorphism with finitely generated kernel we can use Lemma \ref{FibringComposition}. It follows that $\mathrm{SOut}^0(A_\Gamma)$ is algebraically fibred, so $\Out(A_\Gamma)$ virtually fibres.
\end{proof}
Having established the algebraic fibring properties of $Q$, we now turn to its virtual fibring.
\begin{proposition}\label{VirtualFibringQ}
	Assume property (P2) fails. Then $Q$ is virtually fibred if and only if property (P1.n) holds for some $n\geq 2$. Consequently, if property (P1.n) holds for some $n\geq 2$, it follows that $\Out(A_\Gamma)$ is virtually fibred.
\end{proposition}
\begin{proof}
	Let $Q=P\rtimes R$ where $P$ is the subgroup of matrices with all the diagonal blocks equal to the identity and $R$ is the direct product of the diagonal blocks. Take $M\in P$, then we can write $M=E_1\cdots E_r$, where each $E_i$ is an elementary matrix corresponding to a transvection with vertices in two different equivalence classes. By Lemma \ref{ComutatorQ} each $E_i$ lies in $Q'$, so $M\in Q'$ and $P\leq Q'$. Let $L\leq Q$ be a finite index subgroup and assume that there exists a non-trivial homeomorphism $\varphi:L\to\mathbb{Z}$. Since $L\cap P\leq L\cap Q'\leq Q'$ we have that $L\cap P\leq\ker(\varphi)$, so $\varphi$ induces a map $\hat{\varphi}:L\cap R\to\mathbb{Z}$. Assume that $R=\mathrm{SL}_{r_1}(\mathbb{Z})\times\cdots\times \mathrm{SL}_{r_k}(\mathbb{Z})\times \mathrm{SL}_{2}(\mathbb{Z})^n$ with $r_1,\dots,r_k\geq 3$. Then, $\hat{\varphi}=\hat{\varphi}_{r_1}\cdots\hat{\varphi}_{r_k}\hat{\varphi}_n$ where $\hat{\varphi}_{r_i}$ is the restriction of $\hat{\varphi}$ to $L\cap \mathrm{SL}_{r_k}(\mathbb{Z})$ for each $r_i$ and $\hat{\varphi}_n$ is the restriction of $\hat{\varphi}$ to $L\cap  \mathrm{SL}_{2}(\mathbb{Z})^n$.
	Since $\mathrm{SL}_m(\mathbb{Z})$ is not virtually indicable when $m\geq 3$ we have that $\hat{\varphi}_{r_i}$ is trivial for every $r_1,\dots, r_k$.
	Hence, we may assume that $n\geq 1$, otherwise $\hat{\varphi}$ would be trivial and thus $\varphi$ would also be trivial, implying that $Q$ is not virtually fibred. In this case, $\hat{\varphi}$ induces a map $\overline{\varphi}:L\cap \mathrm{SL}_{2}(\mathbb{Z})^n\to\mathbb{Z}$. Observe that $L\cap \mathrm{SL}_{2}(\mathbb{Z})^n$ is of finite index in $ \mathrm{SL}_{2}(\mathbb{Z})^n$. Hence, since $ \mathrm{SL}_2(\mathbb{Z})$ does not virtually fibre, we may assume that $n\geq 2$; otherwise $\overline{\varphi}$ would be trivial and so $\varphi$ would also be trivial, implying again that $Q$ does not virtually fibre.
	
	Recall that $F_2$ is a subgroup of index 12 in $\mathrm{SL}_2(\mathbb{Z})$ (c.f. \cite{Newmann}). Let us define $R_1$ to be the finite index subgroup of $R$ obtained by replacing each factor $\mathrm{SL}_2(\mathbb{Z})$ with its subgroup $F_2$. 
	We define the homeomorphism $\varphi: R_1\to\mathbb{Z}$ by mapping all generators from the $F_2$ factors to $1$ and the remaining generators to zero. By Theorem \ref{Bestvina-Brady}, and since $n\geq 2$, it follows that the kernel of $\varphi$ is finitely generated. Thus, $R$ is virtually fibred. Now, consider the natural projection map $\pi: Q\to R$. The kernel of this map is $P$, which is finitely generated. Therefore, by corollary \ref{VFibringComposition}, it follows that $Q$ is virtually algebraically fibred.
	
	For the last part, since $Q$ is virtually fibred and $\rho:\mathrm{SOut}^0(A_\Gamma)\to Q$ is an epimorphism with finitely generated kernel it follows by Corollary \ref{VFibringComposition} that $\mathrm{SOut}^0(A_\Gamma)$ fibres virtually. Therefore, $\Out(A_\Gamma)$ is virtually fibred.
\end{proof}
In the case when the RAAG $A_\Gamma$ admits no non-inner partial conjugations, we can characterize completely when $\Out(A_\Gamma)$ is virtually fibred:
\begin{corollary}\label{non-inner fibre}Let $A_\Gamma$ be a RAAG that admits no non-inner partial conjugations. Then, $\Out(A_\Gamma)$ virtually fibres if and only if either property (P2) holds or property (P1.n) holds for some $n\geq 2$.
\end{corollary}
\begin{proof}
	Since $A_\Gamma$ admits no non-inner partial conjugations there are also no commutator transvections  [c.f \cite{Sale} Lemma 2.4]. Therefore, the Torelli subgroup $\mathrm{IA}_\Gamma$ is trivial and  we have an isomorphism $\mathrm{SOut}^0(A_\Gamma)\cong Q$. In particular, $Q$ is a finite index subgroup of $\Out(A_\Gamma)$. By Lemma \ref{FibringFiniteIndex} it follows that $\Out(A_\Gamma)$ virtually fibres if and only if $Q$ is virtually fibred. The result then follows directly from Propositions \ref{FibringQ} and \ref{VirtualFibringQ}.
\end{proof}

\begin{remark}\label{Directproductfibring}
	A group $G$ is said to be \textbf{indicable} if it admits a non-trivial homomorphism $\varphi : G \to \mathbb{Z}$. The direct product of indicable groups is algebraically fibred (cf.\ \cite[Theorem 1]{Friedl}), and if $G$ and $H$ are directly indecomposable centerless groups, then $\Aut(G)\times\Aut(H)$ is a finite-index subgroup of $\Aut(G\times H)$ by Proposition~\ref{CenterAutomorphism}. Therefore, using the results on virtual indicability of automorphisms of RAAGs (cf.\ \cite{Aramayona,Sale}), one can construct examples of groups that virtually fibre. We will illustrate this method with explicit examples in Section~\ref{Algebraic fibring examples}.
\end{remark}
\section{$\ell^2$-Betti numbers}
\subsection{Transvection subgroup of $\Out(A_\Gamma)$}\label{Transvections}
Let $A_\Gamma$ be a RAAG. The $\ell^2$-Betti numbers of its transvection group $Q$ can be determined from the associated transvection graph $\Lambda_\Gamma$, defined at the beginning of Section \ref{SectionTransvection}. For the next result we will denote by $\mathrm{loop}(\Lambda_\Gamma)$ the set of all edges in $\Lambda_\Gamma$ that go from a vertex to itself.
\begin{lemma}\label{TranvectionsBetti}
	\begin{enumerate}
		\item If $E(\Lambda_\Gamma)\neq\mathrm{loop}(\Lambda_\Gamma)$ then $\beta_i^{(2)}(Q)=0$ for every $i\geq0$.
		\item If $E(\Lambda_\Gamma)=\mathrm{loop}(\Lambda_\Gamma)$ then $\beta_i^{(2)}(Q)=0$ for every $i\geq0$ unless every equivalence class of vertices has at most $2$ elements. In that case, if there are $n$ equivalence classes with $2$ elements, then:
		$$\beta_k^{(2)}(Q)= \begin{cases} 
			\frac{1}{12^n} & \text{if } $k=n$  \\
			0 & \text{if } k\neq n
		\end{cases}
		$$
	\end{enumerate}
\end{lemma}
\begin{proof}
		Assume first that $E(\Lambda_\Gamma)\neq\mathrm{loop}(\Lambda_\Gamma)$. Let $k$ be the smallest index such that there exists an arrow starting in $V_k$ that is not a loop. Define $Q_k$ to be the normal subgroup of $Q$ generated by the following set of elementary matrices:
		$$\lbrace\mathrm{E}_{i,j}\mid \text{ for each }i,j \text{ such that there is }v_i\in V_{l}, v_j\in V_k\text{ and }  V_k\to V_{l}\text{ in }\Lambda_\Gamma\rbrace$$
		The block structure of $Q$ implies that $Q_k$ corresponds to the subgraph $\Lambda_{\Gamma_k}$, obtained after deleting all arrows not starting in $V_k$. We have, $Q_k=\mathrm{SL}_{\abs{V_k}}(\mathbb{Z})\ltimes \mathrm{Mat}_{\abs{V_k}\times m}(\mathbb{Z})$, where $m=\sum\lbrace \abs{V_j}\mid j>k\text{ and there is an arrow }V_k\to V_j\rbrace$. The semidirect product action is by matrix multiplication. Since $\mathrm{Mat}_{\abs{V_k}\times m}(\mathbb{Z})$ is a non-trivial, free abelian normal subgroup of $Q_k$, all the $\ell^2$-Betti numbers of $Q_k$ vanish. Furthermore, as $Q_k$ is a normal subgroup of $Q$, this implies that all $\ell^2$-Betti numbers of $Q$ also vanish.
		
		Consider the case when $E(\Lambda_\Gamma)=\mathrm{loop}(\Lambda_\Gamma)$, so $Q=\mathrm{SL}_{\abs{V_1}}(\mathbb{Z})\times\cdots\times \mathrm{SL}_{\abs{V_N}}(\mathbb{Z})$. If any equivalence class $V_i$ satisfies $\abs{V_i}\geq3$ then, by Theorem \ref{BettiGL(Z)}, all $\ell^2$-Betti numbers of $\mathrm{SL}_{\abs{V_i}}(\mathbb{Z})$ vanish, and hence so do those of $Q$. Thus, we may assume $Q\cong\mathrm{SL}_2(\mathbb{Z})^n$. The Künneth formula for the $\ell^2$-Betti numbers of the direct product of groups, together with induction over $n$, yields the stated result.
\end{proof}
\subsection{First $\ell^2$-Betti number of $\Aut(A_\Gamma)$ and $\Out(A_\Gamma)$}
As a first step, we compute the zeroth $\ell^2$-Betti number, which is equivalent to determine whether the groups are finite or not.
\begin{lemma}\label{betti0}
	$\beta_0^{(2)}(\Aut(A_\Gamma))>0$ if and only if $A_\Gamma\cong\mathbb{Z}$. Similarly, $\beta_0^{(2)}(\Out(A_\Gamma))>0$ if and only if $A_\Gamma$ admits no non-trival partial conjugations and no transvections.
\end{lemma}
\begin{proof}
	Suppose first that $A_\Gamma$ is not $\mathbb{Z}$. If $A_\Gamma$ is not free abelian then there exists a vertex $v\in V(\Gamma)$ such that $v$ is not central. The inner automorphism given by conjugation by $v$ has infinite order in $\Aut(A_\Gamma)$. If $A_\Gamma$ is free abelian of rank at least two then, for any pair of distinct generators, one can define a transvection which also has infinite order. Since $\Aut(\mathbb{Z})\cong\mathbb{Z}/2\mathbb{Z}$ we conclude that $\beta_0(\Aut(A_\Gamma))>0$ if and only if $A_\Gamma\cong\mathbb{Z}$.  The second statement follows directly from \cite{Wade} Proposition 3.6.
\end{proof}
Now, we compute the first $\ell^2$-Betti number.
\begin{theorem} Let $A_\Gamma$ be a RAAG. Then, $\beta_1^{(2)}(\Aut(A_\Gamma))>0$ if and only if $A_\Gamma=\mathbb{Z}^2$. Similarly, $\beta_1^{(2)}(\Out(A_\Gamma))>0$ if and only if one of the following holds
		\begin{itemize}
			\item  $A_\Gamma$ admits no non-inner partial conjugations, and admits exactly two transvections lying in the same equivalence class; or
			\item $A_\Gamma$ admits no transvections, $\Gamma\setminus\st(v)$ has at most two connected components for every vertex $v\in V(\Gamma)$, and $\mathrm{PSO}(A_\Gamma)$ is isomorphic to  $A_\Theta$ for some disconnected graph $\Theta$.
		\end{itemize}
\end{theorem}
\begin{proof}
	We begin with the $\Aut(A_\Gamma)$ case. By Lemma \ref{betti0} $\Aut(A_\Gamma)$ is finite if and only if $A_\Gamma\cong\mathbb{Z}$, so we may assume  $\abs{V(\Gamma)}\geq 2$. Consider the short exact sequence:
		$$1\rightarrow A_\Gamma/\mathrm{Z}(A_\Gamma)\rightarrow \Aut(A_\Gamma)\rightarrow \Out(A_\Gamma)\rightarrow 1$$
		
		If $A_\Gamma$ is abelian, i.e. $A_\Gamma=\mathbb{Z}^{\abs{V(\Gamma)}}$, then $\Aut(A_\Gamma)=\mathrm{GL}_{\abs{V(\Gamma)}}(\mathbb{Z})$ and the  $\ell^2$-Betti numbers are known (see Theorem~\ref{BettiGL(Z)}). Now, suppose $\mathrm{Z}(A_\Gamma)\neq A_\Gamma$. Then,  $A_\Gamma/\mathrm{Z}(A_\Gamma)$ is infinite and finitely generated, so $\beta_1^{(2)}(A_\Gamma/\mathrm{Z}(A_\Gamma))<\infty$. We consider two cases:
		\begin{itemize}
			\item  If $\abs{\Out(A_\Gamma)}=\infty$ then Proposition \ref{VanishingBetti}.3 implies that $\beta_1^{(2)}(\Aut(A_\Gamma))=0$.
			\item  If $\abs{\Out(A_\Gamma)}<\infty$ then, by Lemma \ref{FiniteOutLemma}, $\Gamma$ is a connected graph and $A_\Gamma$ is a finite index subgroup of $\Aut(A_\Gamma)$. Hence, $\beta_1^{(2)}(A_\Gamma)=0$ if and only if $\beta_1^{(2)}(\Aut(A_\Gamma))=0$. By Theorem \ref{BettiRAAG} $\beta_1^{(2)}(A_\Gamma)=0$ if and only if $\overline{\beta}_0(\widehat{\Gamma})=0$, which is equivalent to $\Gamma$ being connected.
		\end{itemize}
	 So the only case with $\beta_1^{(2)}(\Aut(A_\Gamma)) > 0$ is when $A_\Gamma \cong \mathbb{Z}^2$.
	 
	For the $\Out(A_\Gamma)$ part we analyse different cases:
	\begin{itemize}
		\item If there are no transvections and no partial conjugations, then $\Out(A_\Gamma)$ is finite by Lemma \ref{betti0}, hence $\beta_1^{(2)}(\Out(A_\Gamma))=0$.
		\item  If there are both transvections and partial conjugations consider the short exact sequence:
		$$1\to \mathrm{IA}_\Gamma\to \mathrm{SOut}^0(A_\Gamma)\xrightarrow{\rho} Q\to 1$$
		
		since both $Q$ and $\mathrm{IA}_\Gamma$ are infinite finitely generated non-trivial groups and $\beta_1^{(2)}(\mathrm{IA}_\Gamma)<\infty$ we can apply Proposition \ref{VanishingBetti}.3 Then, it follows that $\beta_1^{(2)}(\Out(A_\Gamma))=0$. 
		\item  If there are transvections but no partial conjugations, there are also no commutator transvections  [c.f \cite{Sale} Lemma 2.4].	Hence, $\mathrm{IA}_\Gamma=1$ and $Q\cong\mathrm{SOut}^0(A_\Gamma)$ is a finite index subgroup of $\Out(A_\Gamma)$. By Lemma \ref{TranvectionsBetti}, $\beta_1^{(2)}(Q)=0$ unless $Q\cong \mathrm{SL}_2(\mathbb{Z})$, which holds precisely when there are exactly two transvections lying in the same equivalence class.
		\item If there are partial conjugations but no transvections, then $\mathrm{PSO}(A_\Gamma)$ is a finite index subgroup of $\Out(A_\Gamma)$. By Proposition \ref{FibringPSO}, we have that $\mathrm{PSO}(A_\Gamma)$ virtually fibres if and only if $\mathrm{PSO}(A_\Gamma)$ is not isomorphic to a RAAG $A_\Theta$ with $\Theta$ disconnected. Since virtual fibring implies vanishing of the first $\ell^2$-Betti number, it follows that $\beta_1^{(2)}(\Out(A_\Gamma))>0$ if and only if  $\mathrm{PSO}(A_\Gamma)$ is isomorphic to a RAAG $A_\Theta$ with $\Theta$ a  disconnected graph (c.f. Corollary \ref{BettiRAAG}).\qedhere
	\end{itemize}
\end{proof}
\subsection{Higher dimensional $\ell^2$-Betti numbers}
In general, computing the $\ell^2$-Betti numbers of $\Aut(A_\Gamma)$ and $\Out(A_\Gamma)$ is a difficult problem. However, for several families of RAAGs, results in the literature imply that all $\ell^2$-Betti numbers of $\Out(A_\Gamma)$ vanish.
\begin{theorem}\label{VanishingHigerBetti}
	Let $A_\Gamma$ be a RAAG. Then, all $\ell^2$-Betti numbers of $\Out(A_\Gamma)$ vanish if any of the following conditions holds:
	\begin{enumerate}
		\item $A_\Gamma\cong\mathbb{Z}^n$ with $n\geq 3$.
		\item $A_\Gamma$ is non-abelian with non-trivial centre.
		\item $A_\Gamma$ admits no non-inner partial conjugations and either $E(\Lambda_\Gamma)\neq\mathrm{loop}(\Lambda_\Gamma)$ or there exists an equivalence class containing at least three vertices.
		\item  $A_\Gamma$ admits non-inner partial conjugations and the graph $\Gamma$ has no SILs.
		\item $\Gamma$ is connected and the link of every non-maximal clique is either discrete or connected. In particular, this includes the case when $\Gamma$ is triangle-free.
		\item $\Gamma$ is connected and contains leaf-like vertices or non-trivial $\hat{v}$-component conjugations for some $v\in V(\Gamma)$. This includes the case where $\Gamma$ contains a leaf.
	\end{enumerate}
\end{theorem}
\begin{proof}
	 All items follow from Theorem \ref{BettiGL(Z)}, [ \cite{Charney} Proposition 4.4], Proposition \ref{TranvectionsBetti}, [\cite{Guirardel} Lemma 3.4.],[\cite{Charney 2} Theorem 6.6] and [ \cite{Charney 2} Theorem 3.4] respectively.
\end{proof}

\section{Examples}
\subsection{First $\ell^2$-Betti number}
From Theorem \ref{Theorem 1.1}, we see that there are two distinct reasons that lead to the non-vanishing of the first $\ell^2$-Betti number of $\Out(A_\Gamma)$. Let us present two examples of RAAGs $A_\Gamma$ such that $\beta_1^{(2)}(\Out(A_\Gamma))>0$, each illustrating one of those behaviours. These examples go beyond the obvious cases where $A_\Gamma\in\lbrace \mathbb{Z}^2,F_2\rbrace$.

\begin{example} \label{VanishingExample}
		Consider the following graph $\Gamma$:
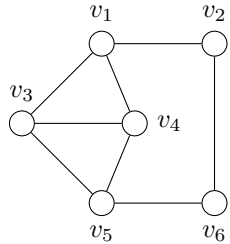
		$$\begin{tikzpicture}[main/.style = {draw, circle},node distance={15mm}] 
			\node[label={$v_3$}][main] (1)  {}; 
			\node[label={$v_1$}][main] (6) [above right of=1] {};
			\node[label=right:{$v_4$}][main] (2) [right of=1] {}; 
			\node[label={$v_2$}][main] (3) [above right of=2] {};
			\node[label=below:{$v_5$}][main] (7) [below right of=1] {};
			\node[label=below:{$v_6$}][main] (8) [below right of=2] {};
			\draw[-] (1) -- (2);
			\draw[-] (1) -- (6);
			\draw[-] (6) --  (2);
			\draw[-] (6) -- (3);
			\draw[-] (8) -- (3);
			\draw[-] (8) -- (7);
			\draw[-] (7) -- (2);	
			\draw[-] (7) -- (1);	
		\end{tikzpicture} $$
	It is straightforward to verify that there are no non-inner partial conjugations and that the only transvections are $\tau_{v_3}^{v_4}$ and $\tau_{v_4}^{v_3}$. Thus, we have $\mathrm{SOut}^0(A_\Gamma)\cong Q\cong\mathrm{SL}_2(\mathbb{Z})$. Since $\abs{\Aut(\Gamma)}=4$ and the subgroup of inversions has order $2^6$, the index of $\mathrm{SOut}^0(A_\Gamma)$ in $\Out(A_\Gamma)$ is $2^6\cdot 4=2^8$.
	Therefore, $\beta_n^{(2)}(\Out(A_\Gamma))=0$ if $n\neq 1$ and:
		$$\beta_1^{(2)}(\Out(A_\Gamma))=\frac{\beta_1^{(2)}(\mathrm{SOut}^0(A_\Gamma))}{\abs{\Out(A_\Gamma):\mathrm{SOut}^0(A_\Gamma))}}=\frac{\beta_1^{(2)}(\mathrm{SL}_2(\mathbb{Z}))}{2^8}=\frac{1}{2^{10}\cdot 3}$$
\end{example}		
		
\begin{example}
		Consider the following graph $\Gamma$:
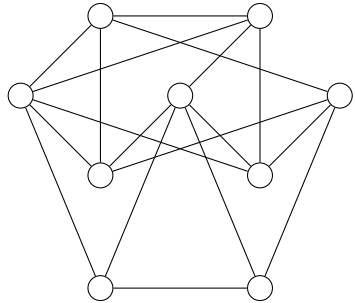
			$$\begin{tikzpicture}[main/.style = {draw, circle},node distance={15mm}] 
			\node[main] (1)  {}; 
			\node[main] (2) [below left of=1] {};
			\node[main] (3) [below right of=1] {}; 
			\node[main] (4) [above right of=3] {};
			\node[main] (5) [below right of=4] {};
			\node[main] (6) [below right of=2] {};
			\node[main] (7) [below right of=3] {};
			\node[main] (8) [below of=6] {};
			\node[main] (9) [below of=7] {};
			\draw[-] (1) -- (2);
			\draw[-] (1) -- (6);
			\draw[-] (1) -- (4);
			\draw[-] (1) -- (5);
			\draw[-] (2) -- (4);
			\draw[-] (2) -- (6);
			\draw[-] (2) -- (7);
			\draw[-] (2) -- (8);
			\draw[-] (3) -- (4);
			\draw[-] (3) -- (6);
			\draw[-] (3) -- (7);
			\draw[-] (3) -- (8);
			\draw[-] (3) -- (9);
			\draw[-] (4) -- (7);
			\draw[-] (5) -- (6);
			\draw[-] (5) -- (7);
			\draw[-] (5) -- (9);
			\draw[-] (8) -- (9);
		\end{tikzpicture} $$
		This graph defines a RAAG $A_\Gamma$ with no transvections and no non-trivial graph automorphisms. In [Appendix A, \cite{Wiedmer}], Wiedmer proves using Theorem \ref{PSO-RAAG} that $\mathrm{PSO}(A_{\Gamma})\cong F_2$. Hence:
		$$\beta_1^{(2)}(\Out(A_{\Gamma}))=\frac{\beta_1^{(2)}(\mathrm{PSO}(A_{\Gamma}))}{\abs{\Out(A_{\Gamma}):\mathrm{PSO}(A_{\Gamma})}}=\frac{\beta_1^{(2)}(F_2)}{2^9}=\frac{1}{2^9}$$		
\end{example}
\subsection{Higher dimensional $\ell^2$-Betti numbers}
Regarding the non-vanishing of $\ell^2$-Betti numbers, we can construct RAAGs whose automorphism or outer automorphism groups have prescribed patterns of non-zero $\ell^2$-Betti numbers. The following two results provide explicit examples illustrating these phenomena.
\begin{lemma}
	For each $n\geq 1$ there exists a graph $\Delta_n$ such that $\beta_i^{(2)}(\Aut(A_{\Delta_n}))>0$ if and only if $i=n$.
\end{lemma}
\begin{proof}
	For $n=1$ take $A_{\Delta_1}\cong\mathbb{Z}^2$, whose automorphism group is $\mathrm{GL}_2(\mathbb{Z})$. For $n\geq 2$, let $\Delta_n$ be the graph $\Gamma_{n-1}$ from Proposition \ref{nsphereRAAG}. By Lemma \ref{FiniteOutLemma}, $A_{\Delta_n}$ is a finite index subgroup of $\Aut(A_{\Delta_n})$. Using Theorem \ref{BettiRAAG} we get that  $\beta_i^{(2)}(A_{\Delta_n})>0$ if and only if $i=n$ and the result follows.
\end{proof}

\begin{proposition}
	Let $\lbrace n_i\rbrace_{i=1}^\infty\subset \lbrace 0,1\rbrace$ be a sequence with all but finitely many terms equal to zero. Then, there exists a RAAG $A_\Gamma$ such that $\beta_i^{(2)}(\Out(A_\Gamma))>0$ if and only if $n_i=1$.
\end{proposition}
\begin{proof}
	For each $i\geq 2$ with $n_i=1$, consider the RAAG $A_{\Gamma_{i-1}}$ from Proposition \ref{nsphereRAAG}, and fix a vertex $v_i\in V(\Gamma_i)$. Construct a new graph $\Delta'$ by identifying all these $v_i$ into a common vertex. If $n_1=0$, set $\Delta=\Delta'$ and if $n_1=1$ define $\Delta$ to be the disjoint union of $\Delta'$ with a single isolated vertex. Then, the flag complex $\widehat{\Delta}$ is homotopy equivalent to the wedge sum of $(i-1)$-spheres for each $i\geq 2$, and possibly a disjoint point if $n_i=1$. Therefore, $\overline{\beta}_{i-1}(\widehat{\Delta})> 0$ if and only if $n_i=1$ . By Theorem \ref{BettiRAAG}, it follows that $\beta_{i}^{(2)}(A_\Delta)>0$ for $i\geq1$ if and only if $n_i=1$. Finally,  by Theorem \ref{FiniteindexRAAG}, there exists a graph $\Gamma$ such that $A_\Delta$ is of finite index in $\Out(A_\Gamma)$, completing the proof.
\end{proof}
\subsection{Algebraic fibring}\label{Algebraic fibring examples}
Let us now consider some examples of RAAGs $A_\Gamma$ for which we can determine whether $\Aut(A_\Gamma)$ or $\Out(A_\Gamma)$ virtually fibre, using the various methods developed throughout this paper.

\begin{example}
		Consider the graph $\Gamma$ given below:
	$$
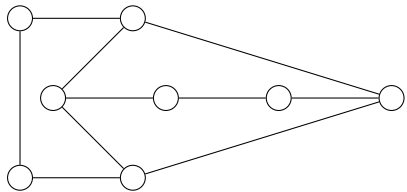
\begin{tikzpicture}[main/.style = {draw, circle},node distance={15mm}] 
		\node[main] (1)  {}; 
		\node[main] (2) [right of=1] {};
		\node[main] (3) [below left of=2] {};
		\node[main] (4) [right of=3] {};
		\node[main] (5) [right of=4] {};
		\node[main] (6) [right of=5] {};
		\node[main] (7) [below right of=3] {};
		\node[main] (8) [left of=7] {};
		\draw[-] (1) -- (2);
		\draw[-] (1) -- (8);
		\draw[-] (2) -- (3);
		\draw[-] (2) -- (6);
		\draw[-] (3) -- (4);
		\draw[-] (3) -- (7);
		\draw[-] (4) -- (5);
		\draw[-] (5) -- (6);
		\draw[-] (6) -- (7);
		\draw[-] (7) -- (8);	
	\end{tikzpicture} $$
 	The RAAG $A_\Gamma$ is transvection-free, so by Corollary \ref{FibringTFAut}, the group $\Aut(A_\Gamma)$ virtually fibres. On the other hand, applying Theorem \ref{PSO-RAAG}, we find that  $\mathrm{PSO}(A_\Gamma)\cong A_{C_4}$ where $C_4$ is the $4$-cycle. Hence, the map sending all generators of $A_{C_4}$ to $1$ defines a fibration of $\mathrm{PSO}(A_\Gamma)$, so $\Out(A_\Gamma)$ also virtually fibres. Moreover, since $\abs{\Aut(\Gamma)}= 4$ and the subgroup of inversions have order $2^8$, the index of $\mathrm{PSO}(A_\Gamma)$ in $\Out(A_\Gamma)$ is $2^8\cdot 4=2^{10}$. Therefore, we obtain:
	$$\beta_2^{(2)}(\Out(A_\Gamma))=\frac{\beta_2^{(2)}(\mathrm{PSO}(A_\Gamma))}{\abs{\Out(A_\Gamma):\mathrm{PSO}(A_\Gamma)}}=\frac{1}{2^{10}}$$
	and all other $\ell^2$-Betti numbers of $\Out(A_\Gamma)$ vanish.
\end{example}
\begin{example}	Consider $\Gamma$ to be the following graph:
	$$
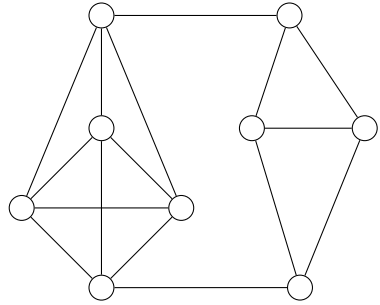
\begin{tikzpicture}[main/.style = {draw, circle},node distance={15mm}] 
		\node[main] (1)  {}; 
		\node[main] (2) [below of=1] {};
		\node[main] (3) [below left of =2] {};
		\node[main] (4) [below right of=2] {};
		\node[main] (5) [below left of=4] {};
		\node[main] (6) [right of=2, xshift=5mm] {};
		\node[main] (7) [right of=6] {};
		\node[main] (8) [right of=1,  xshift=10mm] {};
		\node[main] (9) [below left of=7,  xshift=2mm, yshift=-10.6mm] {};
		\draw[-] (1) -- (2);
		\draw[-] (1) -- (3);
		\draw[-] (1) -- (4);
		\draw[-] (1) -- (8);
		\draw[-] (2) -- (3);
		\draw[-] (2) -- (4);
		\draw[-] (2) -- (5);
		\draw[-] (3) -- (4);
		\draw[-] (3) -- (5);
		\draw[-] (4) -- (5);
		\draw[-] (5) -- (9);
		\draw[-] (6) -- (7);
		\draw[-] (6) -- (8);
		\draw[-] (6) -- (9);
		\draw[-] (7) -- (8);
		\draw[-] (7) -- (9);	
	\end{tikzpicture} $$
	The group $A_\Gamma$ has no non-inner partial conjugations, so $Q$ is a finite index subgroup of $\Out(A_\Gamma)$. One can verify that $E(\Lambda_\Gamma)=\mathrm{loop}(\Lambda_\Gamma)$, and thus $Q\cong\mathrm{SL}_2(\mathbb{Z})\times\mathrm{SL}_3(\mathbb{Z})$. It follows that all $\ell^2$-Betti numbers of $\Out(A_\Gamma)$ vanish. However, $A_\Gamma $ is a RAAG for which property $(P2)$ dos not hold while property $(P1.1)$ does hold. Therefore, by Corollary \ref{non-inner fibre} $\Out(A_\Gamma)$ is not virtually fibred.
\end{example}
\begin{example}	Consider the graph $\Gamma$ given below:
	$$
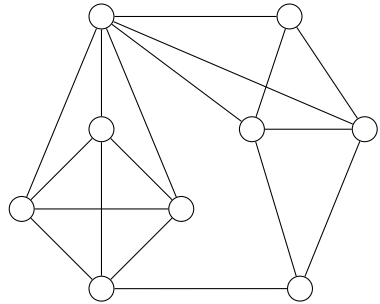
\begin{tikzpicture}[main/.style = {draw, circle},node distance={15mm}] 
		\node[main] (1)  {}; 
		\node[main] (2) [below of=1] {};
		\node[main] (3) [below left of =2] {};
		\node[main] (4) [below right of=2] {};
		\node[main] (5) [below left of=4] {};
		\node[main] (6) [right of=2, xshift=5mm] {};
		\node[main] (7) [right of=6] {};
		\node[main] (8) [right of=1,  xshift=10mm] {};
		\node[main] (9) [below left of=7,  xshift=2mm, yshift=-10.6mm] {};
		\draw[-] (1) -- (2);
		\draw[-] (1) -- (3);
		\draw[-] (1) -- (4);
		\draw[-] (1) -- (8);
		\draw[-] (2) -- (3);
		\draw[-] (2) -- (4);
		\draw[-] (2) -- (5);
		\draw[-] (3) -- (4);
		\draw[-] (3) -- (5);
		\draw[-] (4) -- (5);
		\draw[-] (5) -- (9);
		\draw[-] (6) -- (7);
		\draw[-] (6) -- (8);
		\draw[-] (6) -- (9);
		\draw[-] (7) -- (8);
		\draw[-] (7) -- (9);	
		\draw[-] (1) -- (6);	
		\draw[-] (1) -- (7);	
	\end{tikzpicture} $$
	The group $A_\Gamma$ has no non-inner partial conjugations, so $Q$ is a finite index subgroup of $\Out(A_\Gamma)$. In this case, the transvection graph $\Lambda_\Gamma$ is the following:
		$$
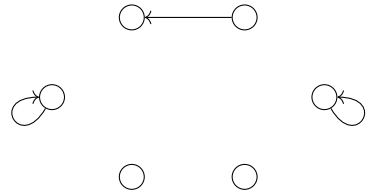
\begin{tikzpicture}[main/.style = {draw, circle},node distance={15mm}] 
		\node[main] (1)  {}; 
		\node[main] (2) [right of=1] {};
		\node[main] (3) [below right of=2] {};
		\node[main] (4) [below left of=3] {};
		\node[main] (5) [below left of=1] {};
		\node[main] (6) [below right of=5] {};
		\draw[->] (3) to [out=300,in=0,looseness=10] (3);
		\draw[->] (5) to [out=240,in=180,looseness=10] (5);
		\draw[->] (2) -- (1);
	\end{tikzpicture} $$
	
	 Since $E(\Lambda_\Gamma)\neq\mathrm{loop}(\Lambda_\Gamma)$, it follows from Lemma \ref{TranvectionsBetti} that all $\ell^2$-Betti numbers of $Q$ vanish. Thus, all $\ell^2$-Betti numbers of $\Out(A_\Gamma)$ vanish. On the other hand, since $A_\Gamma $ satisfies property (P2), $\Out(A_\Gamma)$ fibres by Proposition \ref{FibringQ}.
\end{example}

\begin{example}
	Let $A_\Gamma$ be the RAAG from Example~\ref{VanishingExample}. By \cite[Theorem~1.6]{Aramayona}, $\Aut(A_\Gamma)$ is virtually indicable. Therefore, $\Aut(A_\Gamma \times A_\Gamma)$ virtually algebraically fibres by Remark~\ref{Directproductfibring}.
\end{example}	 
\begin{example}\label{FibringAutLeaf}
	If $\Gamma$ contains a leaf and has trivial centre, then $A_\Gamma$ is directly indecomposable. Hence, if $\Gamma_1$ and $\Gamma_2$ are two graphs each with a leaf and trivial centre, then by \cite[Theorem~1.6]{Aramayona} the groups $\Aut(A_{\Gamma_1})$ and $\Aut(A_{\Gamma_2})$ are virtually indicable. Therefore, $\Aut(A_{\Gamma_1} \times A_{\Gamma_2})$ virtually algebraically fibres by Remark~\ref{Directproductfibring}.
\end{example}
\begin{example}
	Consider the following graph $\Gamma$:
	$$\begin{tikzpicture}[main/.style = {draw, circle},node distance={15mm}] 
		\node[main](1){};
		\node[main] (2) [above left of=1] {}; 
		\node[main] (3) [above right of=1] {}; 
		\node[main] (4) [below left of=1] {};
		\node[main] (5) [below right of=1] {};
		\node[main] (6) [right of=1] {};
		\node[main] (7) [above right of=2] {};
		
		\draw[-] (1) -- (2);
		\draw[-] (2) -- (3);
		\draw[-] (3) -- (1);
		\draw[-] (4) -- (1);
		\draw[-] (5) -- (1);
		\draw[-] (6) -- (1);
		\draw[-] (4) -- (5);
		\draw[-] (7) -- (2);
		\draw[-] (7) -- (3);
	\end{tikzpicture} $$
By Example~\ref{FibringAutLeaf} it follows that $\Aut(A_\Gamma \times A_\Gamma)$ virtually algebraically fibres. Furthermore, the outer automorphism group $\Out(A_\Gamma)$ is virtually indicable by \cite[Theorem~1.6]{Aramayona}. Therefore, $\Out(A_\Gamma \times A_\Gamma)$ virtually algebraically fibres by Remark~\ref{Directproductfibring}. By Proposition~\ref{CenterAutomorphism}, the group $\Out(A_\Gamma)^2$ is a finite-index subgroup of $\Out(A_\Gamma \times A_\Gamma)$. As a result, applying Theorem~\ref{VanishingHigerBetti}.6, we conclude that all $\ell^2$-Betti numbers of $\Out(A_\Gamma)$, and hence of $\Out(A_\Gamma \times A_\Gamma)$, vanish.

\end{example}
\subsection{Higher virtual algebraic fibring}
If $\mathcal{P}$ is a finiteness property, such as being $FP_n$ or $F_n$, a group $G$ is said to be $\mathcal{P}$-\textbf{algebraically fibred} if there exists a group epimorphism $\varphi:G\to\mathbb{Z}$ whose kernel is of type $\mathcal{P}$. This notion is closely connected to $\ell^2$-Betti numbers: if $G$ is virtually $FP_n$-algebraically fibred then $\beta_i^{(2)}(G)=0~\forall~i=0,\dots,n$ [c.f. \cite{Luck} Theorem 7.2 (6)].

We now demonstrate the existence of RAAGs whose (outer) automorphism groups exhibit higher virtual algebraic fibring properties.
\begin{lemma}
	For each $n\geq 1$ there exists a RAAG $A_{\Gamma_n}$  such that $\Aut(A_{\Gamma_n})$ is virtually $FP_n$-algebraically fibred but not $FP_{n+1}$-algebraically fibred.
\end{lemma}
\begin{proof}
	For each $n\geq 1$ consider the RAAG $A_{\Gamma_{n}}$ constructed in Proposition \ref{nsphereRAAG}. This RAAG has finite outer automorphism group and, by Lemma \ref{FiniteOutLemma}, $A_{\Gamma_{n}}$ has finite index in $\Aut(A_{\Gamma_{n}})$. By Theorem \ref{BettiRAAG}, we have $\beta_{n+1}^{(2)}(A_{\Gamma_{n}})>0$, and thus $\beta_{n+1}^{(2)}(\Aut(A_{\Gamma_{n}}))>0$. Hence, $\Aut(A_{\Gamma_{n}})$ is not virtually $FP_{n+1}$-algebraically fibred. On the other hand, by Theorem \ref{Bestvina-Brady}, the associated Bestvina-Brady group is of type $FP_n$, so $A_{\Gamma_{n}}$ is $FP_n$- algebraically fibred. Since $A_{\Gamma_{n}}$ has finite index in its automorphism group,  $\Aut(A_{\Gamma_{n}})$ is virtually $FP_n$-algebraically fibred.
\end{proof}

\begin{lemma} For each $n\geq 2$ the group $\Out(F_2^n)$ is virtually $FP_{n-1}$-algebraically fibred but not $FP_{n}$-algebraically fibred. Moreover, there exists a RAAG $A_\Gamma$ such that $\Out(A_\Gamma)$ is virtually $FP$-algebraically fibred.
\end{lemma}
\begin{proof}
	By  Proposition \ref{CenterAutomorphism} $\Out(F_2)^n$ is of finite index in $\Out(F_2^n)$ and, since $F_2$ has finite index in $\Out(F_2)$  (c.f. \cite{Newmann}), it follows that $F_2^n$ has finite index in $\Out(F_2^n)$. Let $\Gamma_n$ be a graph such that $A_{\Gamma_n}\cong F^{n}_2$. Then, $\widehat{\Gamma_n}=\mathbb{S}^{n-1}$ and Theorem \ref{BettiRAAG} implies that $\beta_{n}^{(2)}(\Out(F_2^n))>0$, so $\Out(F_2^n)$ is not virtually $FP_n$-algebraically fibred. However, the associated Bestvina-Brady group $BB_{\Gamma_n}$ is of type $FP_{n-1}$ by Theorem \ref{Bestvina-Brady}, so $\Out(F_2^n)$ is virtually $FP_{n-1}$-algebraically fibred.
	
	For the second part consider a RAAG $A_{\Omega}$ such that $\widehat{\Omega}$ is contractible. Then, $BB_\Omega$ is of type $FP$ Theorem \ref{Bestvina-Brady} and, by Theorem  \ref{FiniteindexRAAG}, $A_{\Omega}$ embeds as a finite index subgroup of  $\Out(A_\Gamma)$ for some RAAG $A_\Gamma$. This shows that $\Out(A_\Gamma)$ is virtually $FP$-algebraically fibred.
\end{proof}


\begin{thebibliography}{99}
	\bibitem{Abert}
	M. Abért, D. Gaboriau, Higer dimensional cost and profinite actions, in preparation, 2022.
	\bibitem{Aramayona}
	J. Aramayona, C. Martínez-Pérez, On the first cohomology of automorphism groups of graph groups, \textit{Journal of Algebra}, \textbf{452}, 17-41, 2016.
	\bibitem{Atilla}
	 M. Atiyah. Elliptic operators, discrete groups and von Neumannalgebras, \textit{Colloque “Analyse et Topologie” en l’Honneur de Henri Cartan} (Orsay, 1974), pages 43–72. Astérisque, SMF, No. 32–33, Société mathématique de France, Paris, 1976.
	\bibitem{Avramidi}
	G. Avramidi, B. Okun, K. Schreve, Mod p and torsion homology growth in nonpositive curvature, \textit{Inventiones Mathematicae}, \textbf{226}, 711–723, 2021. 
	\bibitem{Bestvina}
	M. Bestvina, N. Brady, Morse theory and finiteness properties of groups, Inventiones mathematicae, \textbf{129}, 445–470, 1997.
	\bibitem{Bieri-Renz}
	R. Bieri, B. Renz, Valuations on Free Resolutions and Higher Geometric Invariants of Groups, \textit{Commentarii Mathematici Helvetici}, \textbf{63}, 464-497, 1988
	\bibitem{Borel}
	A. Borel, The $L^2$-cohomology of negatively curved Riemannian symmetric spaces, \textit{Annales Academiæ Scientiarum Fennicæ}, \textbf{10}, 95–105, 1985.
	\bibitem{Charney 4}
	R. Charney, J. Crisp, K. Vogtmann, Automorphisms of  $2$–dimensional right-angled Artin groups, \textit{Geometry and Topology}, \textbf{11} (4), 2227-2264, 2007.
	\bibitem{Charney}
	R. Charney, K. Vogtmann, Finiteness properties of automorphism groups of right-angled Artin groups, \textit{Bulletin of the London Mathematical Society}, \textbf{41} (1), 94–102, 2009.
	\bibitem{Charney 2}
	R. Charney, K. Vogtmann, Subgroups and quotients of automorphism groups of RAAGs, \textit{Proceedings of Symposia in Pure Mathematics} , \textbf{82}, 9-27, 2011.
	\bibitem{Cheeger}
	J. Cheeger, M. Gromov, L2-Cohomology and group cohomology, \textit{Topology}, \textbf{25} (3), 189-215, 1986.
	\bibitem{Day 3}
	M. B. Day, Peak reduction and finite presentations for automorphism groups of right-angled Artin groups, \textit{Geometry and Topology}, \textbf{13} (2), 817 - 855, 2009.
	\bibitem{Day 2}
	M. B. Day, Symplectic structures on right-angled Artin groups: between the mapping class group and the symplectic group, \textit{Geometry and Topology}, \textbf{13} (2), 857–899, 2009.
	\bibitem{Day 4}
	M. B. Day. R. D. Wade, Subspace arrangements, BNS invariants, and pure symmetric outer automorphisms of right-angled Artin groups, \textit{Groups, Geometry, and Dynamics}, \textbf{12} (1), 173-206, 2015.
	\bibitem{Davis}
	 M. W. Davis, I. J. Leary, The $\ell^2$-cohomology of Artin groups, \textit{Journal of the London Mathematical Society}, \textbf{68} (2), 493–510, 2003.
	\bibitem{Friedl}
	S. Friedl, S. Vidussi,  BNS invariants and algebraic fibrations of group extensions, \textit{Journal of the Institute of Mathematics of Jussieu}, \textbf{22} (2), 985-999, 2023.
	\bibitem{Gaboriau 2}
	D. Gaboriau, Invariants $\ell^2$ de relations d’equivalence et de groupes, \textit{Publications mathématiques de l'IHÉS}, \textbf{95}, 93-150, 2002.
	\bibitem{Gaboriau}
	D. Gaboriau, C. Nôus, On the top dimensional $\ell^2$-Betti numbers, \textit{Annales de la Faculté des sciences de Toulouse : Mathématiques}, \textbf{30} (5), 1121-1137, 2021.
	\bibitem{Guirardel}
	V. Guirardel, A. Sale, Vastness properties of automorphism groups of
	RAAGs, \textit{Journal of Topology}, \textbf{11} (1), 30–64, 2018.
	\bibitem{Kammeyer}
	H. Kammeyer, \textit{Introduction to $\ell^2$-invariants}, Springer, Cham, 2019.
	\bibitem{Kielak}
	D. Kielak, Residually finite rationally solvable groups and virtual fibring, \textit{Journal of the American Mathematical Society}, \textbf{33} (2), 451-486, 2019.
	\bibitem{Koban}
	N. Koban, A. Piggott, The Bieri-Neumann-Strebel invariant of the pure symmetric automorphisms of a right-angled Artin group, \textit{Illinois Journal of Mathematics}, \textbf{58} (1), 27–41, 2014.
	\bibitem{Laurence}
	M. R. Laurence, \textit{Automorphisms of graph products of groups}, PhD thesis, QMW College, University of London, 1992.
	\bibitem{Luck}
	W. Lück, \textit{$L^2$-invariants: theory and applications to geometry and K-theory}, Ergebnisse,  A Series of Modern Surveys in Mathematics, Springer-Verlag, Berlin, 2002.
	\bibitem{Newmann}
	M. Newman, The structure of some subgroups of the modular group, \textit{Illinois Journal of Mathematics}, \textbf{6}, 1962.
	\bibitem{Nielsen}
	J. Nielsen, Die Isomorphismen der allgemeinen, unendlichen Gruppe mit zwei Erzeugenden, \textit{Mathematische Annalen}, \textbf{78}, 385-397, 1918.
	\bibitem{Qiang}
	Z. Qiang, Y. Shengkui, On the bounded index property for products of aspherical polyhedra, \textit{Topological Methods in Nonlinear Analysis}, \textbf{56} (2), 419 - 432, 2020.
	\bibitem{Sale}
	A. W. Sale, On virtual indicability and property (T) for outer automorphism groups of RAAGs, \textit{Groups, Geometry, and Dynamics}, \textbf{18}, 147-190, 2024.
	\bibitem{Servatius}
	H. Servatius. Automorphisms of graph groups, \textit{Journal Algebra}, \textbf{126} (1), 34–60, 1989.
	\bibitem{Strebel}
	R. Strebel, \textit{Notes on the Sigma Invariants}, arXiv preprint, 2012
	\bibitem{Toinet}
	E. Toinet, A finitely presented subgroup of the automorphism group of a right-angled Artin group, \textit{Journal of Group Theory}, \textbf{15} (6), 811–822, 2012.
	\bibitem{Wade}
	R. D. Wade, \textit{Symmetries of free and right-angled Artin groups}, PhD thesis, University of Oxford, 2012.
	\bibitem{Wiedmer}
	M. Wiedmer, Right-angled Artin groups as finite-index subgroups of their outer automorphism groups, Bulletin of the London Mathematical Society, \textbf{56}, 945-958, 2024.

\end{thebibliography}
\end{document}